\documentclass{article}
\usepackage[utf8]{inputenc}

\usepackage{graphicx}

\usepackage{array}
\newcolumntype{H}{>{\setbox0=\hbox\bgroup}c<{\egroup}@{}}

\usepackage{siunitx}

\usepackage{algorithm}
\usepackage{algorithmic}

\usepackage{pgf, tikz}
\usetikzlibrary{shapes,fit,calc}

\usepackage{subcaption} 
\usepackage{hhline}

\usepackage{amsmath}
\usepackage{amsfonts}
\usepackage{geometry}
\usepackage{amssymb}
\usepackage{mathtools}

\usepackage{hhline}

\usepackage{url}
\usepackage{hyperref}

\usepackage{color}

\usepackage{amsthm}
{\theoremstyle {plain}  \newtheorem {theorem} {Theorem}}
{\theoremstyle {plain}  \newtheorem {lemma} {Lemma}}
{\theoremstyle {plain}  \newtheorem {definition} {Definition}}
{\theoremstyle {plain}  \newtheorem {example} {Example}}
{\theoremstyle {plain}  \newtheorem {remark} {Remark}}
{\theoremstyle {plain}  \newtheorem {property} {Property}}

\usepackage[section]{placeins}
\usepackage{authblk}

\title{Orbitopal fixing for the full (sub-)orbitope and application to the Unit Commitment Problem}

\author[1,2]{Pascale Bendotti\thanks{pascale.bendotti@edf.fr}}
\author[2]{Pierre Fouilhoux\thanks{pierre.fouilhoux@lip6.fr}}
\author[1,2]{C\'ecile Rottner\thanks{cecile.rottner@edf.fr}}

\affil[1]{EDF R\&D,  7 Boulevard Gaspard Monge, 
              91120 Palaiseau, France}

\affil[2]{   Sorbonne Universit\'e, 
              LIP6, 
              4 Place Jussieu, 
              75005 Paris, France}

\date{March, 2018}

\begin{document}

\maketitle

\begin{abstract}
This paper focuses on integer linear programs where solutions are binary matrices, and the corresponding symmetry group is the set of all column permutations. 
Orbitopal fixing, as introduced in \cite{Kaibel07}, is a technique designed to break symmetries in the special case of  partitioning (resp. packing) formulations involving matrices with exactly (resp. at most)  one 1-entry in each row.
The main result of this paper is to extend orbitopal fixing to the full orbitope, defined as the convex hull of binary matrices with lexicographically nonincreasing columns.
We determine all the variables whose values are fixed in the intersection of an hypercube face with the full orbitope.
Sub-symmetries arising in a given subset of matrices are also considered, thus leading to define the full sub-orbitope in the case of the sub-symmetric group.
We propose a linear time orbitopal fixing algorithm handling both symmetries and sub-symmetries.
We introduce a dynamic variant 
 of this algorithm where the lexicographical order follows the branching decisions occurring along
the B\&B search. 
Experimental results for the Unit  Commitment Problem are presented. A comparison with state-of-the-art techniques is considered to show the effectiveness of the proposed variants of the algorithm.

\end{abstract}

\section{Definitions}
\label{Defs}

Throughout the paper, we consider an Integer Linear Program (ILP) of the form

\begin{equation}
\label{ILP_orb}
\min \bigg\{ c(x) \; | \; x \in \mathcal{X} \bigg\}, \mbox{ with } \mathcal{X} \subseteq \mathcal{P}(m,n) \mbox{ and } c : \mathcal{P}(m,n)  \rightarrow \mathbb{R}
\end{equation}
where $\mathcal{P}(m,n)$ is the set of $m \times n$ binary matrices.
A symmetry is defined as a permutation $\pi$ of the columns $\{1, ..., n \}$ such that for any solution matrix $x \in \mathcal{X} $, matrix $\pi(x)$ is also solution and has same cost, {\it i.e.}, $\pi(x) \in \mathcal{X} $ and  $c(x) = c(\pi(x))$. The \textit{symmetry group} $\mathcal{G}$ of ILP (\ref{ILP_orb}) is the set of all such permutations.
Symmetry group $\mathcal{G}$ partitions the solution set $\mathcal{X} $ into \textit{orbits}, {\it i.e.}, two matrices are in the same orbit if there exists a permutation in $\mathcal{G}$ sending one to the other.

Symmetries arising in ILP can impair the solution process, in particular when symmetric solutions lead to an excessively large branch and bound (B\&B) search tree (see survey \cite{Margot10}).  Symmetry detection techniques are proposed in \cite{Liberti12,Berthold09}.
Various techniques, so called {\it symmetry-breaking techniques}, are available to handle symmetries in ILP of the form (\ref{ILP_orb}).
The general idea is, in each orbit, to pick one solution, defined as the \textit{representative}, and then restrict the solution set to the set of all representatives.

A technique is said to be {\it full-symmetry-breaking} (resp. {\it partial-symmetry-breaking}) if the solution set is exactly (resp. partially) restricted to the representative set.
Moreover, such a technique may introduce some specific branching rules that interfere with the B\&B search. This can forbid exploiting a user-defined branching rule or, even, the default solver branching settings.
A symmetry-breaking technique is said to be \textit{flexible} if at any node of the B\&B tree, the branching rule can be derived from any linear inequality on the variables. 

Such a technique can be based on specific branching and pruning rules during the B\&B search \cite{Margot03,Ostrowski11}, as well as on symmetry-breaking inequalities \cite{Friedman07,Liberti12,Liberti14} possibly derived from the study of the symmetry-breaking polytope \cite{Hojny17}. Techniques based on symmetry-breaking inequalities are flexible, since they do not rely on a particular B\&B search, and can be full or partial-symmetry-breaking.
Efficient full-symmetry-breaking techniques are usually based on the pruning of the B\&B tree (see survey \cite{Margot10} and computational study \cite{Pfetsch15}) and may also be flexible.

In this article, we focus on a particular symmetry group, the {\it symmetric group} $\mathfrak{S}_n$, which is the group of all column permutations. The most common choice of representative is based on the lexicographical order.
Column $y \in \{0,1\}^m$ is said to be \textit{lexicographically greater than} column $z \in \{0,1\}^m$ if there exists $i \in \{1, ...., m-1\}$ such that
 $\forall i' \leq i$, $y_{i'} = z_{i'}$
and $y_{i+1} > z_{i+1}$, $i.e.$, $y_{i+1} =1$ and $z_{i+1}=0$.
We write $y \succeq z$ if $y$ is equal to $z$ or if $y$ is lexicographically greater than $z$. A matrix $x \in \mathcal{P}(m,n)$ is chosen to be the representative of its orbit if its columns $x(1)$, ..., $x(n)$ are lexicographically non-increasing, $i.e.$, for all $j < n$, $x(j) \succeq x(j+1)$.

The convex hull of all $m \times n$ binary matrices with lexicographically non-increasing columns is called a \textit{full orbitope} and is denoted by $\mathcal{P}_0(m,n)$. 
The solution set $\mathcal{X}$ of ILP (\ref{ILP_orb}) restricted to the set of representatives is then  $\mathcal{P}_0(m,n) \cap \mathcal{X} $.

A complete linear description of the full orbitope is given in \cite{Kaibel10} as an $O(mn^3)$ extended formulation. It is constructed by combining extended formulations of simpler polyhedra. To the best of our knowledge, it has never been used in practice to handle symmetries.  Loos \cite{Loos11} studies orbisacks, $i.e.$, full orbitopes with $n=2$, in an attempt to derive a linear description of the full orbitope in the space of the natural $x$ variables.
Complete linear descriptions of orbisacks, as well as extended formulations, are detailed in \cite{Loos11}.
However, no complete linear description of the full orbitope $\mathcal{P}_0(m,n)$ is known in the $x$ space, and computer experiments conducted in \cite{Loos11} indicate that its facet defining inequalities are extremely complicated.

Special cases of full orbitopes are \textit{packing} and \textit{partitioning orbitopes}, which are restrictions to matrices with at most (resp. exactly) one 1-entry in each row. 
If all matrices in $\mathcal{X} $ have at most (resp. exactly) one 1-entry in each row, then the solution set can be restricted to a packing (resp. partitioning) orbitope.
A complete linear description of these polytopes is given in \cite{Kaibel08}, alongside with a polynomial time separation algorithm.
From this linear description, a symmetry-breaking algorithm, called \textit{orbitopal fixing}, is derived in \cite{Kaibel10} in order to consider only the solutions included in the packing (resp. partitioning) orbitope during the B\&B search. It is worth noting that orbitopal fixing is flexible, full-symmetry-breaking and does not introduce any additional inequalities. These key features make orbitopal fixing for packing and partitioning orbitopes particularly efficient.

In this article, we propose an orbitopal fixing algorithm for the full orbitope. In the case of an arbitrary symmetry group, a fixing algorithm can be used to break symmetries during the B\&B search, such as isomorphism pruning \cite{Margot01}, orbital branching \cite{Ostrowski11} or strict setting \cite{Margot03,OstrowPhD}.
When the solution set can be restricted to the full orbitope, the authors in \cite{Ostrowski15} introduce \textit{modified orbital branching} (MOB) which is an efficient partial-symmetry-breaking technique.
In SCIP 5.0 \cite{SCIP} a heuristic similar to MOB takes into account some symmetries related to the full symmetric group.

There are many problems whose symmetry group is the symmetric group acting on the columns, or on a subset of the columns, but whose solution space cannot be restricted to a partitioning or a packing orbitope.
Examples range from line planning problems in public transports \cite{Borndorfer07} to scheduling problems with a discrete time horizon, like the Unit Commitment Problem. The Min-up/min-down Unit Commitment Problem (MUCP) is to find a minimum-cost power production plan on a discrete time horizon for a set of production units. At each time period, the total production has to meet a forecast demand. Each unit must satisfy minimum up-time and down-time constraints besides featuring production and start-up costs. In practical instances, there are several sets of identical units. Assuming a solution is expressed as a matrix where column $j$ corresponds to the up/down trajectory of unit $j$ over the time horizon, then any permutation of columns corresponding to identical units leads to another solution with same cost. In this case, the columns are partitioned into $h$ subsets $\mathcal{N}_1$, ..., $\mathcal{N}_h$, and each subset $\mathcal{N}_k$ contains $n_k$ columns, corresponding to $n_k$ identical units.
The corresponding symmetry group is a product of symmetric groups $\mathfrak{S}_{n_1} \times \mathfrak{S}_{n_2} \times ... \times \mathfrak{S}_{n_h}$, such that $\mathfrak{S}_{n_k}$ is operating on column subset $\mathcal{N}_k$.
In this article we focus on these symmetries which are a priori known. No detection techniques are therefore used.
In \cite{Ostrowski15}, the MOB technique is used on MUCP instances with additional technical constraints.

The orbitopal fixing algorithm for the full orbitope proposed in this article handles the symmetries related to the symmetric group arising in the aforementioned problems. 
This is a flexible full-symmetry-breaking technique which is computationally efficient. 
We generalize symmetries and full orbitopes to a given set of matrix subsets, thus introducing sub-symmetries and sub-orbitopes. Such subsets appear in particular as underlying subproblems of a B\&B search. 
The main motivation to look at sub-symmetries is that they are often undetected in the symmetry group $\mathcal{G}$ of the problem. We extend our orbitopal fixing algorithm to break sub-symmetries arising in sub-symmetric groups.
%Note that it does not increase the size of the LP solved at each node of the B\&B tree.

The paper is organized as follows.
Section \ref{sec:soa} introduces state-of-the-art techniques to handle symmetries in the B\&B tree when the solution set $\mathcal{X}$ is a set of binary matrices, and the symmetry group $\mathcal{G}$ is the symmetric group acting on the columns. 
In Section \ref{sec:intersection}, we characterize the smallest cube face containing the binary intersection of the full orbitope $\mathcal{P}_0(m,n)$ with any face of the $(m,n)$-dimensional 0/1 cube.
In Section \ref{sec:sub} we introduce sub-symmetries and study conditions guaranteeing that at least one optimal solution is preserved by sub-symmetry-breaking techniques. When considering a set of sub-symmetric groups, the lexicographical order qualifies, thus leading to the definition of full sub-orbitope.
In Section \ref{sec:OF} we describe two variants of a linear time 
orbitopal fixing algorithm for the full (sub)-orbitope. The first variant is referred to as static, as it is defined for the natural lexicographical order. The second variant is referred to as dynamic, as the lexicographical order follows the branching decisions occurring along the B\&B search. 
Finally, in Section \ref{sec:MUCP}, numerical experiments are performed on MUCP instances featuring identical production units. A comparison with state-of-the-art symmetry-breaking techniques (Cplex and MOB) is presented in order to show
the effectiveness of our approach.

\section{Handling symmetries in the B\&B tree}
\label{sec:soa}

At some point in the B\&B tree, there will be variables whose values are fixed as a result of the previous branching decisions.
Given a node $a$ of the B\&B tree, let $B^a_1$ (resp. $B^a_0$) be the set of indices of variables fixed to 1 (resp. 0) at node $a$.
Pruning strategies can be constructed using dedicated branching rules or variable fixing algorithms, which will take symmetries into account in order to avoid exploring some symmetric solutions in the B\&B tree.

\subsection{Isomorphism pruning and pruning with branching rules}
\label{sec:MOB-rules}

For a general symmetry group $\mathcal{G}$, Margot introduces in \cite{Margot01} isomorphism pruning, which is to prune any node $a$ in the B\&B tree such that all solutions to subproblem $a$ are not representatives. This can be done provided specific choices of branching variables are made throughout the B\&B tree, thus restricting flexibility. In \cite{Margot03}, a more flexible branching rule for isomorphism pruning is defined.
In \cite{Ostrowski11,Ostrowski08}, a flexible partial-symmetry-breaking technique called  
orbital branching (OB) is introduced.
It is possible to couple isomorphism pruning or orbital branching with an additional variable setting algorithm \cite{Margot03,OstrowPhD}. This procedure sets to 0 (resp. 1) the variables which, at a given node, are symmetric to a variable already set to 0 (resp. 1).

When handling only all column permutation symmetries, 
modified orbital branching (MOB) has been introduced in \cite{Ostrowski15} as a variant of orbital branching in order to produce more balanced B\&B trees. 

At each node $a$ of the B\&B tree, modified orbital branching \cite{Ostrowski15} is to branch on a disjunction that fixes a larger number of variables than the classical disjunction $x_{i,j} = 0 \; \vee \; x_{i,j} = 1$ does. For a given variable $x_{i,j}$ selected for branching at node $a$, consider the set $O_{i,j}(a)$ of all variables which could permute values with $x_{i,j}$ at node $a$: $O_{i,j}(a)$ is defined as the set of all variables $x_{i,j'}$, $j' \leq n$, such that for all $i' \leq m$, $x_{i',j'}$ and $x_{i',j}$ are either fixed to the same value ($i.e.$, $\{(i',j), \;(i',j')\} \subset B^a_1$ or $\{(i',j), \;(i',j')\} \subset B^a_0$) or free ($i.e.$, $\{(i',j), \;(i',j')\} \subset \{1, ..., m \} \times \{1, ...., n \} \backslash  (B^a_0 \cup B^a_1)$). 
Thus, for a given variable orbit $O_{i,j}(a) = \{ x_{i,j_1}, x_{i,j_2}, ..., x_{i,j_k} \}$ at node $a$, and for a given $\alpha \in \mathbb{N}$, modified orbital branching is to branch on the following disjunction:
$$
 x_{i, j_\ell} = 1, \; \forall \ell \in \{1, ..., \alpha \} \quad \vee \quad  x_{i, j_\ell} = 0, \; \forall \ell \in \{\alpha, ...,k \}
$$
 
Modified orbital branching is a partial-symmetry-breaking technique. In the case when the symmetry group is $\mathfrak{S}_n$,
Ostrowski \textit{et al.} \cite{Ostrowski15} show that modified orbital branching can be enforced to a full-symmetry breaking technique.
To this end, one need to apply an additional branching rule restricting the variable orbits which can be branched on at each node.
Namely, the authors introduce the minimum row-index (MI) branching rule, stating that variable $x_{i,j}$ is eligible for branching if and only if for all rows $i' < i$, variables $x_{i',j}$ have already been fixed.
They prove that modified orbital branching alongside with MI branching rule is full-symmetry-breaking.
As the MI rule makes MOB non-flexible, they also propose some relaxed rules for which the full-symmetry-breaking property still holds. In particular, a more flexible branching rule is the \textit{relaxed minimum-rank index} (RMRI).
For each MOB variant, the full-symmetry-breaking property is obtained at the expense of flexibility.

Isomorphism pruning and MOB have similar actions on the B\&B tree. Thus only MOB is considered in the following, as it is dedicated to symmetries arising from the symmetric group

\subsection{Pruning with variable fixing}

Let $C^{(m,n)}$ be the $(m,n)$-dimensional 0/1-cube.
Given an ILP of the form (\ref{ILP_orb}), consider a polytope $P \subset C^{(m,n)}$ such that the solution set of (\ref{ILP_orb}) is a subset of $P$. At a given node $a$ of the B\&B tree, some variables have been already fixed by previous branching decisions. Additional variable fixings can be performed on some of the remaining free variables. The idea is to fix to 0 (resp. 1) variables that would yield a solution outside $P$ if fixed to 1 (resp. 0).
Variable fixing methods, introduced in \cite{Kaibel08}, are presented as follows.

A non-empty face $F$ of $C^{(m,n)}$ is given by two index sets $I_0$, $I_1 \subset \{1, ...., m\} \times \{1, ..., n \}$ such that
$$
F = \{ x \in C^{(m,n)} \; | \; x_{i,j} = 0 \; \forall (i,j) \in I_0 \mbox{ and }  x_{i,j} = 1 \; \forall (i,j) \in I_1 \}.
$$

For a polytope $P \subset C^{(m,n)}$ and a face $F$ of $C^{(m,n)}$ defined by $(I_0, I_1)$, the smallest face of $C^{(m,n)}$ that contains $P \cap F \cap \{0,1\}^{(m,n)}$ is denoted by $\mbox{Fix}_F(P)$, $i.e.$, $\mbox{Fix}_F(P)$ is the intersection of all faces of $C^{(m,n)}$ that contain $P \cap F \cap \{0,1\}^{(m,n)}$.
If $\mbox{Fix}_F(P)$ is a non-empty face of $C^{(m,n)}$, the index sets defining it will be denoted by $I_0^{\star}$ and $I_1^{\star}$. As pointed out in \cite{Kaibel07}, the following result can be directly derived from the definition of $\mbox{Fix}_F(P)$.

\begin{lemma}
\label{fixing}
If $\mbox{Fix}_F(P)$ is a non-empty face, then $\mbox{Fix}_F(P)$ is given by sets $I_0^{\star}$ and $I_1^{\star}$ such that
$$
I_0^{\star} = \bigg\{ (i,j) \; | \; x_{i,j} = 0 \; \;  \forall x \in P \cap F \cap \{0,1\}^{(m,n)} \bigg\}
$$
$$
I_1^{\star} = \bigg\{ (i,j) \; | \; x_{i,j} = 1 \;\;  \forall x \in P \cap F \cap \{0,1\}^{(m,n)} \bigg\}
$$
\end{lemma}

When solving ILP (\ref{ILP_orb}) by B\&B, to each node $a$ corresponds a face $F(a)$ of $C^{(m,n)}$ defined by $B^a_0$ and $B^a_1$.
The aim of variable fixing is then to find, at each node $a$, sets $I_0^{\star}$ and $I_1^{\star}$ defining $\mbox{Fix}_{F(a)}(P)$, where $P$ is a given polytope containing the solution set. 
From Lemma \ref{fixing}, there are two cases. If $\mbox{Fix}_{F(a)}(P) = \varnothing$ then $P \cap F(a) \cap \{0,1\}^{(m,n)}$ is empty as well and node $a$ can be pruned. If $\mbox{Fix}_{F(a)}(P) \not= \varnothing$, then, any 
 free variable in $I_0^{\star}$ (resp. $I_1^{\star}$) can be set to 0 (resp. 1).
 
From Lemma \ref{fixing}, any variable $x_{i,j}$ such that $
(i,j) \not\in I_0^{\star} \cup I_1^{\star}$ cannot be fixed, as it takes either value 0 or 1 in solution subset $P \cap F(a) \cap \{0,1\}^{(m,n)}$. It proves that the fixings occur as early as possible in the B\&B tree.

In general, the problem of computing $\mbox{Fix}_F(P)$ is NP-hard. However, if one can optimize a linear function over $P \cap \{0,1\}^{(m,n)}$ in polynomial time, the fixing ($I_0^{\star}$, $I_1^{\star}$) of the face defined by given index subsets $(I_0, I_1)$ can be computed in polynomial time \cite{Kaibel07} by solving $2(mn - |I_0| - |I_1|)$ many linear optimization problems over $P \cap \{0,1\}^{(m,n)}$.
%If sets $I_0^{\star}$ and $I_1^{\star}$ relative to $P$ cannot be computed efficiently, some relaxations of $P$ can be considered.Instead of computing $\mbox{Fix}_F(P)$, one may only compute $\mbox{Fix}_F(P')$ where $P \subset P'$.

\paragraph{}

Orbitopal fixing is variable fixing with polytope $P$ being an orbitope.
It corresponds to the case when
 the solution set $\mathcal{X}$ of ILP (\ref{ILP_orb}) is restricted to an orbitope $P$. The resulting solution set $\mathcal{X} \cap P$ is trivially included in $P$.
Then variable fixing can be performed in order to restrict the solution set at each node $a$ to be included in orbitope $P$.

In \cite{Kaibel07}, the authors characterize the sets $I_0^{\star}$ and $I_1^{\star}$ defining $\mbox{Fix}_F(P)$ where $P$ is the partitioning (or packing) orbitope and where $F$ is defined by $(B^a_0, B^a_1)$, at a given node $a$ of the B\&B tree.
They derive a linear time orbitopal fixing algorithm performed at each node of the B\&B tree, breaking all orbitopal symmetries from packing and partitioning formulations. 

Orbitopal fixing is a full-symmetry-breaking technique which is also flexible, as fixing does not interfere with branching.
Note also that no additional inequalities are required, thus it does not increase the size of the LP solved at each node.

In the next two sections, we devise a linear time orbitopal fixing algorithm for the full orbitope.

\section{Intersection with the full orbitope}
\label{sec:intersection}

For convenience, the full orbitope $\mathcal{P}_0(m,n)$ is denoted by $P_O$ in this section.
Given a face $F$ of $[0,1]^{(m,n)}$ defined by sets $(I_0, I_1)$,
we will characterize the sets $I_0^{\star}$ and $I_1^{\star}$ defining the fixing $\mbox{Fix}_F(P_O)$ of the full orbitope at $F$. 
Note that face $F$ can be chosen arbitrarily. %In the special case of a face $F(a)$ defined by index subsets $(B^a_0, B^a_1)$ at a node $a$ of the B\&B tree, the corresponding fixing sets $I_0^{\star}$ and $I_1^{\star}$ will indicate all the variables that can be fixed to 0 or 1 at node $a$.

We first define $F(P_O)$-minimality and $F(P_O)$-maximality, which are key properties for matrices. Namely we will see that each column $j$ of an $F(P_O)$-minimal (resp. $F(P_O)$-maximal) matrix is the lexicographically lowest (resp. greatest) possible $j^{th}$ column of any binary matrix $X \in P_O \cap F \cap \{0,1\}^{(m,n)}$.

For any matrix $X$, the $j^{th}$ column of $X$ is denoted by $X(j)$ and the entry at row $i$, column $j$ by $X(i,j)$ .

\begin{definition}
For a given face $F$ of $[0,1]^{(m,n)}$, a matrix $X$ is said to be $F(P_O)$-minimal (resp. $F(P_O)$-maximal) if $X \in P_O \cap F \cap \{0,1\}^{(m,n)}$ and for any matrix $Y \in P_O \cap F \cap \{0,1\}^{(m,n)}$, $X(j)$ is lexicographically less (resp. greater) than or equal to $Y(j)$, $\forall j \in \{ 1, ..., n \}$, $i.e.$, $X(j) \preceq Y(j)$ (resp. $X(j) \succeq Y(j)$) $\forall j \in \{ 1, ..., n \}$.
\end{definition}
The section is organized as follows.

\begin{itemize}
\item[] \begin{itemize}
    \item[1.] Two sequences of matrices $(\underline{{M}}^j)_{j \in \{1, ..., n \}}$ and $(\overline{{M}}^j)_{j \in \{1, ..., n \}}$ are introduced, such that
    matrices $\underline{{M}}^1$ and $\overline{{M}}^n$ will respectively be $F(P_O)$-minimal and $F(P_O)$-maximal.
    \item[2.] Sets $I_1^{\star}$ and $I_0^{\star}$ are determined from $\underline{{M}}^1$ and $\overline{{M}}^n$.
    
    %all variables that have value 1 (resp. 0) in every $X \in P_O \cap F \cap \{0,1\}^{(m,n)}$ are determined, thus characterizing $I_1^{\star}$ (resp. $I_0^{\star}$).
\end{itemize}
\end{itemize}

We now introduce some definitions.
Some matrices considered in this section are partial matrices in the sense that their entries can take values in the set $\{0, 1, \times\}$, where $\times$ represents a free variable.
A given partial matrix $M$ of size $(m,n)$ is fully given by the pair $(S_0, S_1)$ of index subsets such that the indices corresponding to a 0-entry in $M$ are in subset $S_0$ and the indices corresponding to a 1-entry in $M$ are in subset $S_1$. The remaining indices $\{1, ..., m \} \times \{1, ...., n \} \backslash (S_0 \cup S_1)$ correspond to free variables in $M$.

For a given column $j \in \{1, ..., n-1\}$, the following definitions are useful to compare columns $j$ and $j+1$ of matrix $M$.

\begin{definition} \textcolor{white}{b}

\begin{itemize}
\item[$\bullet$] 
A row $i \in \{ 1, ..., m\}$ is said to be ${j}$-\textit{fixed}, for a given $j <n$, if $M(i,j) \not= \times$ and $M(i,j+1) \not= \times$ and $M(i,j) \not= M(i,j+1)$. 

%Columns $j$ and $j+1$ are said to be \textit{fixedly different} on a given row $i$ if $M(i,j) \not= \times$ and $M(i,j+1) \not= \times$ and $M(i,j) \not= M(i,j+1)$. 

Let $i_f(M,j)$ be the smallest ${j}$-fixed row in $\{ 1, ..., m\}$, if such a row exists, and $m+1$ otherwise.

\item[$\bullet$] 
%Columns $j$ and $j+1$ are said to be \textit{discriminated} on a given row $i$ if $M(i, j) \not= 0$ and $ M(i, j+1) \not= 1$.  Let $i_d(M,j)$ be the largest row index in $\{ 1, ..., i_f(M,j)\}$ where $j$ and $j+1$ are discriminated, if such a row exists, and 0 otherwise.
A row $i$ is said to be ${j}$-\textit{discriminating}, for a given $j <n$, if $M(i, j) \not= 0$ and $ M(i, j+1) \not= 1$. 
 
 Let $i_d(M,j)$ be the largest ${j}$-discriminating row in $\{ 1, ..., i_f(M,j)\}$ if such a row exists, and 0 otherwise.
\end{itemize}
\end{definition}

%\begin{definition}Consider two columns $j$ and $j'$ of a matrix $X_{\mathcal{S}}$. Let index $i^{\mathcal{S}}_c(j, j')$ be the index of the first row of $X_\mathcal{S}$ where columns $j$ and $j'$ are \textit{potentially different}, $i.e.$:If columns $j$ and $j'$ are fixed to the same values on each row then $i^{M'}_c(j, j')$ is arbitrarily set to $m+1$. Otherwise,\begin{itemize}\item for each row $i < i^{M'}_c(j, j')$, entries $X_{M'}(i, j)$ and $ X_{M'}(i, j')$ are fixed to the same value,\item at row $i^{M'}_c(j, j')$, columns $j$ and $j'$ are either fixed to different values, or at least one is free.\end{itemize}\end{definition}

\begin{example}
\label{exa}
To illustrate, consider matrix $M'$ defined by pair $(S'_0, S'_1)$, with
$S'_0 = \{ (4,1),  (3,2),  (5,2)\}$ and $S'_1 = \{(2,1), (5,1), (4,2), (1, 3), (2, 3)\}$:
\begin{center}
$M' =$ $\begin{bmatrix} 
\times & \times & 1\\
1 & \times & 1 \\
\times & 0 & \times  \\
0  & 1 & \times  \\
1 & 0& \times 
\end{bmatrix}$
\end{center}
%Columns 1 and 2 are fixedly different only on rows 4 and 5.
%There is no row on which columns 2 and 3 are fixedly different, so $i_f(M',2)=6$.
Only rows 4 and 5 are 1-fixed. 
Hence $i_f(M',1)=4$. 
There is no $2$-fixed row, so $i_f(M',2)=6$.
Rows 1, 2, 3 and 5 are 1-discriminating, hence $i_d(M',1) = 3$. Only row 4 is 2-discriminating then $i_d(M',2) = 4$.
\end{example}

\subsection{Matrix sequences $(\underline{{M}}^j)_{j \in \{1, ..., n \}}$ and $(\overline{{M}}^j)_{j \in \{1, ..., n \}}$}

We propose an algorithm constructing a sequence of matrices $(\underline{{M}}^j)_{j \in \{1, ..., n \}}$ (resp. $(\overline{{M}}^j)_{j \in \{1, ..., n \}}$) of size $(m,n)$. For each $j$, matrix $\underline{{M}}^j$ (resp. $\overline{{M}}^j$) will be derived from pair  $(\underline{S}^j_0, \underline{S}^j_1)$ (resp. $(\overline{S}^j_0, \overline{S}^j_1)$).
Matrices $\underline{{M}}^1$ and $\overline{{M}}^n$ will respectively be $F(P_O)$-minimal and $F(P_O)$-maximal if $\mbox{Fix}_F(P_O)$ is non-empty. Otherwise, they will be arbitrarily defined by the sets
$S^{\varnothing}_0 = \{ (1, 1) \}$, 
$S^{\varnothing}_1  = \{ 1, ...., m \} \times \{ 1, ..., n \} \backslash S^{\varnothing}_0 $.

\begin{algorithm}

\caption{Construction of sequence $(\underline{{M}}^j)_{j \in \{1, ..., n \}}$ defined by pair $(\underline{S}^j_0, \underline{S}^j_1)_{j \in \{1, ..., n \}}$}
\begin{algorithmic} 
\normalsize
 \STATE $j \leftarrow n$. 
 
 \STATE $\underline{S}^n_1 \leftarrow I_1 $
\STATE $\underline{S}^n_0 \leftarrow \{ (i,n) \not\in I_1 \} \cup I_0 $

 \FOR{$j=n-1$ to 1}
    \STATE  $i_f \leftarrow i_f(\underline{M}^{j+1},j)$
   
    \IF{$i_f = m+1$}
        \STATE $
\underline{S}^j_1 \leftarrow \underline{S}^{j+1}_1 \cup \bigg\{ (i,j) \not\in \underline{S}^{j+1}_0 \; | \; (i, j+1) \in \underline{S}^{j+1}_1 \bigg\}
$

\STATE $
\underline{S}^j_0 \leftarrow \underline{S}^{j+1}_0 \cup \bigg\{ (i,j) \not\in  \underline{S}^{j+1}_1 \; | \; (i,j+1) \in \underline{S}^{j+1}_0 \bigg\}
$

    \ELSIF{there is no $j$-discriminating row $i \in \{ 1, ..., i_f \}$ in matrix $\underline{{M}}^{j+1}$}
    
        \STATE $(\underline{S}^{j'}_0, \underline{S}^{j'}_1) \leftarrow (S^{\varnothing}_0,S^{\varnothing}_1)$, $ \forall j' \leq j$
    \ELSE
        \STATE
$
i_{d} \leftarrow i_d(\underline{{M}}^{j+1}, j)
$

\STATE 
$
\underline{S}^j_1 \leftarrow \underline{S}^{j+1}_1 \cup \{ (i_{ld}, j) \} \cup \bigg\{ (i,j) \not\in \underline{S}^{j+1}_0  \; | \; (i, j+1) \in \underline{S}^{j+1}_1  \mbox{ and } i < i_{ld} \bigg\}
$

$
\underline{S}^j_0 \leftarrow \underline{S}^{j+1}_0 \cup \bigg\{  (i,j) \not\in \underline{S}^j_1 \bigg\}.
$

    \ENDIF

 \ENDFOR
\end{algorithmic}
\label{algo}
\end{algorithm}

The key idea for the construction of matrix sequence $(\underline{{M}}^j)_{j \in \{1, ..., n \}}$ is the following. 
For $j=n$, matrix $\underline{{M}}^{n}$ is defined by pair $(I_0, I_1)$, except that each free variable in column $n$ is set to 0.
For each $j<n$, matrix $\underline{{M}}^{j}$ is defined to be equal to matrix $\underline{{M}}^{j+1}$, except that
free variables in column $\underline{{M}}^{j+1}(j)$ are set to 0 or 1 in matrix $\underline{{M}}^{j}$. This is done by propagating values from column $j+1$, so that column $j$ is minimum among all columns greater than or equal to column $j+1$. Note that in matrix $\underline{{M}}^{j}$, there are no remaining free variables in columns $\{j, ..., n\}$.

%From this construction, matrix $\underline{{M}}^{1}$ will be shown to be $F(P_O)$-minimal by induction. Indeed, for any $X \in P_O \cap {F} \cap \{0,1\}^{(m,n)}$, $\underline{{M}}^n(n) \preceq X(n)$. If for a given $j<n$, $\underline{{M}}^{j+1}(j+1) \preceq X(j+1)$ holds for any $X \in P_O \cap {F} \cap \{0,1\}^{(m,n)}$,  $\underline{{M}}^{j}$ has been constructed such that $\underline{{M}}^{j}(j) \preceq X(j)$.

The construction of sequence $(\underline{{M}}^j)_{j \in \{1, ..., n \}}$ is given in Algorithm \ref{algo}.
For $j=n$, matrix $\underline{{M}}^{n}$ is defined by pair $(I_0 \cup \{ (i,n) \not\in I_1 \}, I_1)$. For $j<n$, if $i_f(\underline{M}^{j+1},j) = m+1$ then each free variable in column $\underline{{M}}^j(j)$ is set such that columns $j$ and $j+1$ are equal. Otherwise, there are two cases. In the first case, $i_f(\underline{M}^{j+1},j) \leq m$ and there is no $j$-discriminating row $i \in \{ 1, ..., i_f \}$ in matrix $\underline{{M}}^{j+1}$. Then for all $j' \leq j$, $(\underline{S}^{j'}_0, \underline{S}^{j'}_1)$ is set to $(S^{\varnothing}_0,S^{\varnothing}_1)$.
In the second case, $i_f(\underline{M}^{j+1},j) \leq m$ and there exists a $j$-discriminating row $i \in \{ 1, ..., i_f \}$ in matrix $\underline{{M}}^{j+1}$. Let row $i_{d} = i_d(\underline{{M}}^{j+1}, j)$. Free variables in column $\underline{{M}}^j(j)$ are set such that columns $j$ and $j+1$ are equal from row 1 to row $i_{d} -1$, and such that row $i_{d}$ has the form $[1, \; 0 ]$ on columns $j$ and $j+1$.
Every other free variable in column $j$ is set to 0.

As the definition of sequence $(\overline{{M}}^j)_{j \in \{1, ..., n \}}$ is very similar, the corresponding algorithm is omitted.
For $j=1$, matrix $\overline{{M}}^1$ is defined by pair $(I_0 , I_1 \cup \{ (i,1) \not\in I_0 \})$.
For $j>1$, free variables in column $\overline{{M}}^{j-1}(j)$ are set to 0 or 1 in matrix $\overline{{M}}^{j}$ by propagating values from column $j-1$, so that column $j$ is maximum among all columns less than or equal to column $j-1$. 

Referring to $(S'_0, S'_1)$ defined in Example \ref{exa} alongside with matrix $M'$, corresponding matrix sequence $(\underline{{M}}^k)_{k\in \{1,2,3\}}$ is as follows.
\begin{center}
$\underline{{M}}^3 =$ $\begin{bmatrix} 
\times & \times & 1\\
1 & \times & 1 \\
\times & 0 & 0 \\
0  & 1 & 0  \\
1 & 0 & 0 
\end{bmatrix}$,
\quad
$\underline{{M}}^2 =$ $\begin{bmatrix} 
\times & 1 & 1\\
1 & 1 & 1 \\
\times & 0 & 0 \\
0  & 1 & 0  \\
1 & 0 & 0 
\end{bmatrix}$,
\quad
$\underline{{M}}^1 =$ $\begin{bmatrix} 
1 & 1 & 1\\
1 & 1 & 1 \\
1 & 0 & 0 \\
0  & 1 & 0  \\
1 & 0& 0 
\end{bmatrix}$.
\end{center}

The first 2-fixed row in matrix $\underline{{M}}^3$ is row 4. Row 4 is also the last 2-discriminating row in matrix $\underline{{M}}^3$ before row 4. Thus $i_f(M',2)=4$, $i_d(M',2) = 4$ and variables (1,2) and (2,2) in matrix $\underline{{M}}^2$ are set to be equal the corresponding values in column $\underline{{M}}^3(2)$. 
The first 1-fixed row in matrix $\underline{{M}}^2$ is row 4. The last 1-discriminating row before row 4 in matrix $\underline{{M}}^2$ is row $i_{d}(\underline{{M}}^2, 1)=3$. Since $i_{d}(\underline{{M}}^2, 1)=3$, entries (1,1) and (3,1) are set to 1 in matrix $\underline{{M}}^1$.
Matrix sequence $(\overline{{M}}^k)_{k\in \{1,2,3\}}$ is obtained similarly. Finally, for any matrix $X$ in the face defined by $(S'_0, S'_1)$, Theorem \ref{thm:sequence} shows that the following inequalities hold column-wise:

\begin{center}
$\underline{{M}}^1 =$ $\begin{bmatrix} 
1 & 1 & 1\\
1 & 1 & 1 \\
1 & 0 & 0 \\
0  & 1 & 0  \\
1 & 0& 0 
\end{bmatrix}$
\quad
$\preceq$
\quad
$X$
\quad
$\preceq$
\quad
$\overline{{M}}^3 =$ $\begin{bmatrix} 
1 & 1 & 1\\
1 & 1 & 1 \\
1 & 0 & 0  \\
0  & 1 & 1  \\
1 & 0& 0 
\end{bmatrix}$
\end{center}

\iffalse
\begin{center}
$\underline{{M}}^1 =$ $\begin{bmatrix} 
1 & 1 & 1\\
1 & 1 & 1 \\
1 & 0 & 0 \\
0  & 1 & 0  \\
1 & 0& 0 
\end{bmatrix}$
\quad
$\preceq$
\quad
$M'$
\quad
$\preceq$
\quad
$\overline{{M}}^3 =$ $\begin{bmatrix} 
1 & 1 & 1\\
1 & 1 & 1 \\
1 & 0 & 0  \\
0  & 1 & 1  \\
1 & 0& 0 
\end{bmatrix}$
\end{center}

Note that entry $(3,1)$ in matrix $\underline{{M}}^1$ is set to 1. Indeed, as $\underline{M}^{2}$ inherits from $M'$, $i_f(\underline{M}^{2},1) = 4$ and the only 1-discriminating row $i \in \{ 1, ..., 4\}$ in matrix $\underline{M}^{2}$ is row 3. Thus, if entry $(3,1)$ were set to 0 matrix $\underline{{M}}^1$ would not have lexicographically non-increasing columns. As $\overline{M}^{2}$ inherits from $M'$, $i_f(\overline{M}^{2},2) = 6$, free variables in column $\overline{{M}}^2(3)$ are set such that columns 2 and 3 are equal in matrix $\overline{{M}}^3$. It follows that entries $(3,3)$ and $(5,3)$ are set to 0 and entry $(4,3)$ to 1 in matrix $\overline{{M}}^3$.
\fi

\begin{theorem}
\label{thm:sequence}

If $(\underline{S}^1_0, \underline{S}^1_1)   = (S^{\varnothing}_0,S^{\varnothing}_1)$ or 
 $ (\overline{S}^n_0, \overline{S}^n_1)  = (S^{\varnothing}_0,S^{\varnothing}_1)$ then $\mbox{Fix}_{{F}}(P_O) = \varnothing$.
Otherwise matrix $\underline{{M}}^1$ is ${{F}}(P_O)$-minimal and matrix  $\overline{{M}}^n$ is ${{F}}(P_O)$-maximal.

\iffalse
\begin{itemize}
\item[(i)]
\item[(ii)] Similarly, if $\overline{\mathcal{S}}^n = \mathcal{S}^{\varnothing}$ then $\mbox{Fix}_{{F}}(P_O) = \varnothing$.

Otherwise, if $\overline{\mathcal{S}}^n \not= \mathcal{S}^{\varnothing}$, $\mbox{Fix}_{{F}}(P_O) \not= \varnothing$ and matrix $\overline{{M}}^n$ is $\mbox{Fix}_{{F}}(P_O)$-maximal.
\end{itemize}
\fi
\end{theorem}

\begin{proof}
We will prove that if $(\underline{S}^j_0, \underline{S}^j_1)  \not= (S^{\varnothing}_0,S^{\varnothing}_1)$, then, $\forall X \in \mbox{Fix}_{{F}}(P_O)$, $\forall j \in \{1, ..., n\}$,
    $\underline{{M}}^{j}(j) \preceq X(j)$, and otherwise $\mbox{Fix}_{{F}}(P_O) = \varnothing$.
A similar proof can be done to obtain the corresponding result for $(\overline{S}^n_0, \overline{S}^n_1)$ and $\overline{{M}}^j$.
The property is proved by induction on decreasing index value $j \in \{1, ..., n\}$.

For $j=n$, by construction $(\underline{S}^n_0, \underline{S}^n_1)  \not= (S^{\varnothing}_0,S^{\varnothing}_1)$. Since all $(i,n) \not\in I_1$ are set to 0 in matrix $\underline{{M}}^n$, necessarily $\forall X \in P_O \cap {F} \cap \{0,1\}^{(m,n)}$,
    $\underline{{M}}^n(n) \preceq X(n)$.
    
Suppose the induction hypothesis holds for $j+1$, with $j<n$. There are two cases: either $(\underline{S}^j_0, \underline{S}^j_1)  \not= (S^{\varnothing}_0,S^{\varnothing}_1)$ or not.

On the one hand, suppose $(\underline{S}^j_0, \underline{S}^j_1)  \not= (S^{\varnothing}_0,S^{\varnothing}_1)$. Suppose also there exists $X \in P_O \cap {F} \cap \{0,1\}^{(m,n)}$ such that
    $\underline{{M}}^j(j) \succ X(j)$.
    Consider the first row $i$ such that columns $X(j)$ and $\underline{{M}}^j(j)$ are different. As $\underline{{M}}^j(j) \succ X(j)$, we have $X(i,j) = 0$ and $\underline{{M}}^j(i,j)= 1$.
    By construction, since $(i,j) \not\in I_0 \cup I_1$ and $\underline{{M}}^j(i,j)= 1$, for all $i' < i$, $ \underline{{M}}^j(i',j) = \underline{{M}}^j(i',j+1)$.
If $\underline{{M}}^j(i,j+1) = 1$, then since $\underline{{M}}^j(j+1)  = \underline{{M}}^{j+1}(j+1), \underline{{M}}^{j+1}(j+1) \succ X(j)$. 
    By the induction hypothesis, $X(j+1) \succeq \underline{{M}}^{j+1}(j+1)$ thus $X(j+1) \succ X(j)$, which contradicts $X \in P_O$.
Let now suppose $\underline{{M}}^j(i,j+1) = 0$, then, from the construction of $\underline{{M}}^j$, row $i_f = i_f(\underline{M}^{j+1},j)$ in matrix $\underline{{M}}^{j+1}$ has the form $[0, \; 1 ]$ on columns $j$ and $j+1$ (otherwise $\underline{{M}}^j(i,j)$ would have been set to 0).
        In this case, row $i$ corresponds to the last $j$-discriminating row of matrix $\underline{{M}}^{j+1}$ before row $i_f$. Thus, for each $i' \in \{ i+1, i_f-1 \}$ such that $(i', j) \not\in I_0 \cup I_1$, we have $\underline{{M}}^j(i',j+1) = 1$. If for such an $i'$, $X(i',j) = 0$ then since $\underline{{M}}^j(j+1)  = \underline{{M}}^{j+1}(j+1)$, $\underline{{M}}^{j+1}(j+1) \succ X(j)$. Otherwise, as row $i_f$ in matrix $\underline{{M}}^{j+1}$ has the form $[0, \; 1 ]$ on columns $j$ and $j+1$, it follows $(i_f,j) \in I_0$, thus $X(i_f,j)=0$. Consequently 
        $\underline{{M}}^{j+1}(j+1) \succ X(j)$ holds too.
    By the induction hypothesis, $X(j+1) \succeq \underline{{M}}^{j+1}(j+1)$ thus we reach the same contradiction.
    
On the other hand, suppose $(\underline{S}^j_0, \underline{S}^j_1) = (S^{\varnothing}_0,S^{\varnothing}_1)$, consider the following two cases:
 If $(\underline{S}^{j+1}_0, \underline{S}^{j+1}_1) = (S^{\varnothing}_0,S^{\varnothing}_1)$ then by the induction hypothesis, $\mbox{Fix}_{{F}}(P_O) = \varnothing$.
 Otherwise, $(\underline{S}^{j+1}_0, \underline{S}^{j+1}_1) \not= (S^{\varnothing}_0,S^{\varnothing}_1)$.
Recall $i_f = i_f^{}(\underline{M}^{j+1},j)$.
Then, by construction of matrix $\underline{{M}}^j$, row $i_f$ of matrix $\underline{{M}}^{j+1}$ has the form $[0, \; 1 ]$ on columns $j$ and $j+1$ and there is no row $i \in \{1, ..., i_f - 1 \}$ in matrix $\underline{{M}}^{j+1}$ which is $j$-discriminating.
As column $j+1$ is completely fixed in matrix $\underline{{M}}^{j+1}$, each row $i \in \{ 1, ..., i_f - 1 \}$ of matrix $\underline{{M}}^{j+1}$ has one the following forms on columns $j$ and $j+1$: $[1, \; 1 ]$ or $[0, \; 0 ]$ or $[\times, \; 1 ]$. Therefore, if $\mbox{Fix}_{{F}}(P_O)$ were not empty, then $P_O \cap {F} \cap \{0,1\}^{(m,n)} \not= \varnothing$ and for any  $X \in P_O \cap {F} \cap \{0,1\}^{(m,n)}$,
even if $X(i,j) = 1$ for each $(i,j) \not\in I_0 \cup I_1$, $\underline{{M}}^{j+1}(j+1) \succ X(j)$ would hold. By the induction hypothesis, $X(j+1) \succeq \underline{{M}}^{j+1}(j+1)$ thus $X(j+1) \succ X(j)$, which contradicts $X \in P_O$.
\qed \end{proof}

\subsection{Determining $I_0^{\star}$ and $I_1^{\star}$}
In case $\mbox{Fix}_{{F}}(P_O) \not= \varnothing$, sets $I_0^{\star}$ and $I_1^{\star}$ can be characterized using ${{F}}(P_O)$-minimal and ${{F}}(P_O)$-maximal matrices $\underline{{M}}^{1}$ and $\overline{{M}}^n$ as follows.
For each $j \in \{ 1, ..., m \}$, consider row $i_j$, the first row at which columns  $\underline{{M}}^1(j)$ and $\overline{{M}}^n(j)$ differs, defined as:
$$
i_j = \min \bigg\{ i \in \{1, ..., m\} \; | \;  \underline{{M}}^1(i,j) \not= \overline{{M}}^n(i,j)
 \bigg\}
$$

If columns $\underline{{M}}^1(j)$ and $\overline{{M}}^n(j)$ are equal, then $i_j$ is arbitrarily set to $m+1$.
By definition of ${{F}}(P_O)$-minimal and ${{F}}(P_O)$-maximal matrices, $\underline{{M}}^1(i_j,j) < \overline{{M}}^n(i_j,j)$.
Note that since for all $(i,j) \in I_0$ (resp. $I_1$), $\underline{{M}}^1(i,j) = \overline{{M}}^n(i,j) = 0$ (resp. 1), it follows that
$(i_j,j)$ is a free variable $i.e.$, $(i_j,j) \not\in I_0 \cup I_1$ .

\begin{theorem} 
\label{I0I1}
$\mbox{Fix}_{{F}}(P_O)$, if non-empty, is given by sets
$
I^{\star}_0 = I_0 \cup I^+_0
$
and
$
I^{\star}_1 = I_1 \cup I^+_1
$, where
$$
\begin{array}{ll}
  I_0^+ = \bigg\{ (i,j) \not\in I_0 \cup I_1 \; | \; i < i_j \mbox{ and } \overline{{M}}^n(i,j) = 0 
\bigg\},  \qquad & \qquad I_1^+ = \bigg\{ (i,j) \not\in I_0 \cup I_1  \; | \; 
 i < i_j \mbox{ and } \underline{{M}}^1(i,j) = 1
\bigg\}. \\
\end{array}
$$
\end{theorem}

\begin{proof}
($\implies$) We prove that $I^+_0 \subset I_0^{\star}$ and $I^+_1 \subset I_1^{\star}$.
Let us suppose the opposite: $I^+_0 \not\subset I_0^{\star}$ or $I^+_1 \not\subset I_1^{\star}$. Let $(i,j) \in (I^+_0 \backslash I_0^{\star}) \cup (I^+_1 \backslash I_1^{\star})$.
Consider $i_0 = \min \{ i' \; | \; (i',j) \in (I^+_0 \backslash I_0^{\star}) \cup (I^+_1 \backslash I_1^{\star} ) \} $.
Suppose $(i_0,j) \in I^+_0 \backslash I_0^{\star}$. The proof is similar if we suppose $(i_0,j) \in I^+_1 \backslash I_1^{\star}$.
As $(i_0,j) \not\in I_0^{\star}$, there exists $X \in P_O \cap {F} \cap \{0,1\}^{(m,n)}$ such that $X(i_0, j) = 1$.
As $(i_0,j) \in I^+_0$, $\overline{{M}}^n(i_0,j) = 0 $.
If for all $i' < i_0$, $X(i',j) \geq \overline{{M}}^n(i',j)$ then the following would hold: $X(j) \succ \overline{{M}}^n(j)$, contradicting the fact that $\overline{{M}}^n$ is ${F}(P_O)$-maximal.
Thus, there exists a row $i_1 < i_0$ such that $\overline{{M}}^n(i_1,j) = 1 $ and $X(i_1,j) = 0$.
Note that as $(i_0,j) \in I^+_0$, $i_0 < i_j$, and thus $i_1 < i_j$, which implies $\underline{{M}}^1(i_1,j) = 1 $.
Thus $(i_1,j) \in I^+_1$. However, $(i_1,j) \not\in I_1^{\star}$ because $X \in P_O \cap {F} \cap \{0,1\}^{(m,n)}$ and $X(i_1,j) = 0$.
The contradiction comes from the fact that $i_1 < i_0$ and $ i_1 \in \{ i' \; | \; (i',j) \in (I^+_0 \backslash I_0^{\star}) \cup (I^+_1 \backslash I_1^{\star} ) \}$.
This proves $I^+_0 \subset I_0^{\star}$ and $I^+_1 \subset I_1^{\star}$, thus $ I_0 \cup I^+_0 \subset I^{\star}_0 $ and $ I_1 \cup I^+_1 \subset I^{\star}_1$.

($\impliedby$) We prove $I^{\star}_0 \subset I_0 \cup I^+_0$ and $I^{\star}_1 \subset I_1 \cup I^+_1$.
It suffices to show that given $(i,j) \not\in I^{\star}_0 \cup I^{\star}_1$, there exists a solution $X_0 \in P_O \cap {F} \cap \{0,1\}^{(m,n)}$ such that $X_0(i, j) = 0$ and a solution $X_1 \in P_O \cap {F} \cap \{0,1\}^{(m,n)}$ such that $X_1(i, j) = 1$.
Consider index $(i_j,j) \not\in I_0 \cup I_1$. 
Solution $\underline{{M}}^1$ is such that $\underline{{M}}^1(i_j,j) = 0$ and solution $\overline{{M}}^n$ is such that $\overline{{M}}^n(i_j,j) = 1$. 
So if $i = i_j$, the result is proved.
Now suppose $i \not= i_j$.
Note that for all $i' < i_j$, $\underline{{M}}^1(i',j) = \overline{{M}}^n(i',j)$, therefore $(i',j) \in I^{\star}_0 \cup I^{\star}_1$.
Thus $i > i_j$.
Consider solutions $X_0$ and $X_1$ defined as follows. For each  $i'\in \{1, ..., m\}$ and $j'\in \{1, ..., n\}$,
$$
X_0(i',j') = \left\{
\begin{array}{ll}
   \overline{{M}}^n(i',j') & \mbox{if } j' < j\\
   \underline{{M}}^1(i',j') & \mbox{if } j' > j\\
   \underline{{M}}^1(i',j') & \mbox{if } j'=j \mbox{ and } i' < i_j \\ 
   1 & \mbox{if } j'=j \mbox{ and } i'=i_j \\
   0 & \mbox{ otherwise}.
\end{array}
\right.
\qquad X_1(i',j') = \left\{
\begin{array}{ll}
   \overline{{M}}^n(i',j') & \mbox{if } j' < j\\
   \underline{{M}}^1(i',j') & \mbox{if } j' > j\\
   \underline{{M}}^1(i',j') & \mbox{if } j'=j \mbox{ and } i' < i_j \\ 
   0 & \mbox{if } j'=j \mbox{ and } i'=i_j \\
   1 & \mbox{ otherwise}.
\end{array}
\right.
$$

Recall that $ \overline{{M}}^n(i_j, j) = 1$ and $\underline{{M}}^1(i_j,j) = 0$, therefore $ \overline{{M}}^n(j) \succeq X_0(j) \succ X_1(j) \succ \underline{{M}}^1(j) $.
As $ \overline{{M}}^n$ and $ \underline{{M}}^1 \in P_O$, $ \overline{{M}}^n(j-1) \succeq \overline{{M}}^n(j)$  and $\underline{{M}}^1(j) \succeq \underline{{M}}^1(j+1)$.
Thus $X_1$ and $X_ 0 $ are also in $P_O \cap {F} \cap \{0,1\}^{(m,n)}$ and are such that  $X_1(i,j) = 1$ and $X_0(i,j) = 0$. This concludes the proof.
\qed \end{proof}

To illustrate, consider matrices $\underline{{M}}^1$ and $\overline{{M}}^3$ from Example \ref{exa}. Here the rows $i_j$ are respectively $i_1=6$, since $\underline{{M}}^1(1)= \overline{{M}}^3(1)$, $i_2=6$ and $i_3= 4$.
The corresponding sets $I_0^+$ and $I_1^+$ are
$
I^+_0 = \{ (3,3) \}
$
 and
$
I^+_1 = \{ (1,1), \; (3,1), \; (1,2), \; (2,2) \}
$.
Note that indices $(4,3)$ and $(5,3)$ are neither in $I^+_0$ nor in $I^+_1$ because they belong to rows greater than or equal $i_3=4$. The associated variables cannot be fixed, even though variable $x(5,3)$ is set to 0 in $\underline{{M}}^1$ and $\overline{{M}}^3$ .

\paragraph{}

Matrices $\underline{{M}}^1$ and $\overline{{M}}^n$ can be computed in $O(mn)$ time, since at each iteration $j \in \{1, ..., n \}$ of Algorithm \ref{algo}, $i_f$ and $i_{ld}$ can be computed in $O(m)$ time. Once matrices $\underline{{M}}^1$ and $\overline{{M}}^n$ are known, sets $I^{\star}_0$ and $I^{\star}_1$ can be computed in $O(mn)$ time as well. It follows:

\begin{theorem}
\label{lineartime}
 $\mbox{Fix}_{{F}}(P_O)$ can be computed in linear time $O(mn)$.
\end{theorem}

\section{Sub-symmetries and sub-orbitopes} \label{sec:sub}

In the next sections, we generalize symmetries and full orbitopes to a given set of matrix subsets. 
A similar notion has been introduced in the context of Constraint Satisfaction Programming \cite{Gent05,Gent07}.
Symmetries corresponding to such subsets can be detected and tackled during the B\&B search. 

In this section, the symmetry group of an ILP is the set of all index permutations $\pi$ (and not only column permutations) such that for any solution matrix $X \in \mathcal{X} $, matrix $\pi(X)$ is also solution and has same cost, {\it i.e.}, $\pi(X) \in \mathcal{X} $ and  $c(X) = c(\pi(X))$.

\subsection{Sub-symmetries}
Consider a subset ${Q} \subset \mathcal{X}$ of solutions of ILP (\ref{ILP_orb}).
The sub-symmetry group $\mathcal{G}_{Q}$ relative to subset ${Q}$ is defined as the symmetry group of subproblem $\min \{ c x \; | \; x \in {Q} \}$.
For instance, such subset ${Q} \subset \mathcal{X}$ can correspond to a B\&B node, defined as solutions satisfying branching inequalities.

Permutations in sub-symmetry group  $\mathcal{G}_{Q}$  are referred to as \textit{sub-symmetries}. 
The main motivation to look at sub-symmetries in $\mathcal{G}_{Q}$ is that they remain undetected in the symmetry group $\mathcal{G}$ of the problem.
This is illustrated in the following example.

\begin{example}
\label{ex_sub}
Consider an ILP whose solution set is $\mathcal{X} = \{ X_1, \; X_2, \; Y \} \subset \{0, 1\}^{(1,3)}$, where

$$
X_1 = [1, \; 0, \; 1],\quad X_2 = [1, \; 1, \; 0],\quad Y = [0, \; 1, \; 0].
$$

Suppose also solutions $X_1$ and $X_2$ have same cost, $i.e.$, $c(X_1) = c(X_2)$.
Consider solution subset $Q \subset \mathcal{X}$ such that $ Q = \{ X \in \mathcal{X} \; | \; X(1,1)+X(1,2)+X(1,3)=2 \}$, then $Q = \{ X_1, X_2 \}$.
Now consider transposition $\pi_{132}$ such that $\pi_{132}(X) = [X(1,1), X(1,3), X(1,2)]$.
Obviously, $\pi_{132}$ is in sub-symmetry group $\mathcal{G}_Q$, but not in symmetry group $\mathcal{G}$, as $\pi_{132}(Y) = [0, \; 0, \; 1] \not\in \mathcal{X}$. 
\end{example}

\begin{property}
Two solutions in the same orbit with respect to a sub-symmetry group $\mathcal{G}_{Q}$ may not be in the same orbit with respect to the symmetry group $\mathcal{G}$.
\end{property}

Referring to Example \ref{ex_sub}, solutions $X_1$ and $X_2$ are in the same orbit with respect to $\mathcal{G}_{Q}$ since $\pi_{132} \in \mathcal{G}_{Q}$. To see that solutions $X_1$ and $X_2$ are not in the same orbit with respect to $\mathcal{G}$, it is sufficient to show that there is no permutation $\pi \in \mathcal{G}$ such that $\pi(X_1) = X_2$. 
Suppose there was such a permutation $\pi$. First note that $\pi_{132} \not\in \mathcal{G}$ thus $\pi \not= \pi_{132}$. Since both $X_1$ and $X_2$ have exactly one entry to 0, $\pi$ must be such that $\pi(e_2) = e_3$, where, for $i \in \{1, ..., 3\}$ $e_i \in \{0, 1\}^{(1,3)}$ is such that $e_i(1,i) = 1$ and $e_i(1,j) = 0$, $\forall j \not= i$. Since $Y = e_2$, $\pi(Y) = e_3 \not\in \mathcal{X}$, which is a contradiction. Thus, $X_1$ and $X_2$ are not in the same orbit with respect to the symmetry group $\mathcal{G}$, which shows the symmetry acting between these two solutions is not detected in $\mathcal{G}$.

%There may nevertheless be a more general permutation $\pi \in \mathcal{G}$ detecting the symmetry between $X_1$ and $X_2$, satisfying for example $\pi(X_1) = X_2$.

\paragraph{}

We now generalize to sub-symmetries the concepts introduced for symmetries in Section \ref{Defs}.

Let $\{ Q_i \subset \mathcal{X}, \; i \in \{1, ..., s\}\}$ be a set of matrix subsets. To each $Q_i$, $i \in \{1, ..., s\}$, corresponds a sub-symmetry group $\mathcal{G}_{Q_i}$ containing sub-symmetries that may not be detected in the symmetry group $\mathcal{G}$.
Let  $O^i_k$, $k \in \{1, ..., o_i\}$, be the orbits defined by $\mathcal{G}_{Q_i}$ on subset ${Q_i}$, $i \in \{1, ..., s\}$. 

When considering only the symmetry group $\mathcal{G}$, the orbits of the solutions form a partition of the solution set $\mathcal{X}$.
However, the set $\mathcal{O} = \{ O^i_k, \;
k \in \{1, ..., o_i\}, \; i \in \{1, ..., s\} \}$ of orbits defined by sub-symmetry groups $\mathcal{G}_{Q_i}$, $i \in \{1, ..., s\}$, does not form a partition of $\mathcal{X}$  anymore.
Indeed, for given $i, \; j \in \{1, ..., s\}$, if $Q_i \cap Q_j  \not= \varnothing$, then any $x \in Q_i \cap Q_j$ will appear in both the orbits of $\mathcal{G}_{Q_i}$ and the orbits of $\mathcal{G}_{Q_j}$.
In order to break such sub-symmetries, removing all non-representatives of an orbit of $\mathcal{G}_{Q_i}$ may remove the representative of an orbit of $\mathcal{G}_{Q_j}$, thus leaving the latter unrepresented.
%Therefore, if a representative $r(\sigma)$ is associated to each orbit $\sigma \in \mathcal{O}$, when all non-representative solutions of $\sigma$, $i.e.$, solutions in  $\cup_{\sigma \in \mathcal{O}} \sigma \backslash r(\sigma)$, are removed from the solution set, it may actually remove the representative of another orbit $\sigma'$, leaving it unrepresented.

We thus generalize the concept of orbit in order to define a new partition of $\mathcal{X}$ taking sub-symmetries into account.
First, for given $X \in \mathcal{P}(m,n)$, let us define $\mathcal{G}(X)$, the set of all permutations $\pi$ in $\bigcup_{i=1}^s \mathcal{G}_{Q_i}$ such that $\pi$ can be applied to $X$:
$$
\mathcal{G}(X) = \bigcup_{Q_i \ni X} \mathcal{G}_{Q_i}
$$

We now define a relation $\mathcal{R}$ over the solution set $\mathcal{X}$.
Matrix $X'$ is said to be in relation with $X$, written $X' \; \mathcal{R}\; X$, if
$$
 \exists r \in \mathbb{N}, \; \exists \pi_1, ..., \pi_r \; | \; \pi_k \in \mathcal{G}(\pi_{k-1} ... \pi_1(X)), \forall  k \in \{1, ..., r\}, \mbox{ and } \; X' = \pi_1 \pi_2 ... \pi_{r} (X).
$$

The \textit{generalized orbit $\mathbb{O}$ of $X$ with respect to $\{ Q_i, i \in \{1, ..., s\} \}$}  is thus the set of all $X'$ such that $X' \; \mathcal{R}\; X$.
Roughly speaking, orbits that intersect one another are collected into generalized orbits.
Matrix $X'$ is in the generalized orbit of $X$  if $X'$ can be obtained from $X$ by composing permutations of groups $\mathcal{G}_{Q_i}$, ensuring that each permutation $\pi \in \mathcal{G}_{Q_i}$ is applied to an element of $Q_i$.
Note that $\mathcal{R}$ is an equivalence relation, thus the set of all generalized orbits with respect to $\{ Q_i, i \in \{1, ..., s\} \}$ is a partition of $\mathcal{X}$. Moreover, for a given $X \in \mathcal{X}$, each $X'$ in the generalized orbit of $X$ is such that $X' \in \mathcal{X}$ and $c(X') = c(X)$.
Note that the generalized orbit may not be an orbit of any of the symmetry groups $\mathcal{G}_{Q_i}$, $i \in \{1, ..., s\}$.

\begin{remark}
By definition, for any generalized orbit $\mathbb{O}$, there exist orbits $\sigma_1$, ..., $\sigma_{p} \in \mathcal{O}$ such that 
$\mathbb{O}= \cup_{i=1}^{p} \sigma_i$.
\end{remark}

Note that the union $\mathbb{O}= \cup_{i=1}^{p} \sigma_i$ may contain several orbits relative to the same subset $Q_i$.

\begin{example}
\label{ex_gen_orbit}
Consider an ILP having the following feasible solutions:
$$
X_1 = [1, \; 1, \; 0, \; 0],\quad X_2 = [1, \; 0, \; 0, \; 1],\quad X_3 = [0, \; 0, \; 1, \; 1], \quad X_4 = [0, \; 1, \; 1, \; 0], \quad X_5 = [0, \; 1, \; 0, \; 1]
$$
$$
\quad Y_1 = [1, \; 0, \; 0, \; 0],\quad Y_2 = [0, \; 0, \; 0, \; 1].
$$
with $c(X_1) = c(X_2) = c(X_3) = c(X_4) = c(X_5)$ and $c(Y_1) = c(Y_2)$.

Let $Q_1 = \{X_1, \; X_2, \; X_3, \; X_4 \}$, $Q_2 = \{X_1, \; X_5\}$, $Q_3 = \{X_4, \; X_5\}$ and $Q_4 = \{Y_1, \; Y_2\}$.
The permutation sending $X$ to $[X(1,j_1), X(1,j_2), X(1,j_3), X(1,j_4)]$ is denoted by $\pi_{j_1 j_2 j_3 j_4}$.
Note that $\pi_{2341} \in \mathcal{G}_{Q_1}$, $\pi_{4231} \in \mathcal{G}_{Q_2}$, $\pi_{1243} \in \mathcal{G}_{Q_3}$ and $\pi_{4231} \in \mathcal{G}_{Q_4}$.
Thus, the generalized orbit of $X_1$ with respect to $\{ Q_1, Q_2, Q_3, Q_4\}$ is $\{X_1, X_2, X_3, X_4, X_5 \}$, as $X_2 = \pi_{2341}(X_1)$, $X_3 = \pi_{2341} (X_2)$, $X_4 = \pi_{2341}  (X_3)$ and $X_5 = \pi_{1243} (X_4)$. Similarly, the generalized orbit of $Y_2$ with respect to $\{ Q_1, Q_2, Q_3, Q_4\}$ is $\{Y_1, Y_2 \}$.  All in all there are two generalized orbits  $O = \{X_1, X_2, X_3, X_4, X_5\}$ and $O' = \{Y_1, Y_2 \}$. Note that $O'$ corresponds to the single orbit $Q_4$.
\end{example}

While simple orbits $\sigma \in \mathcal{O}$ may sometimes be easily determined, the generalized orbits may anyway be difficult to compute. 
In this case, one may want to choose a representative $r(\sigma) \in \sigma$ for each orbit $\sigma \in \mathcal{O}$, and then use a \textit{sub-symmetry-breaking} technique to remove all elements $\sigma \backslash r(\sigma)$ from the search, for each orbit $\sigma \in \mathcal{O}$. As for given orbit $\sigma$, the set $\sigma \backslash r(\sigma)$ may contain the representative of another orbit $\sigma'$, 
we need to ensure that there remains at least one element per generalized orbit after the removal of all elements $\cup_{\sigma \in \mathcal{O}}( \sigma \backslash r(\sigma))$. To this end the choice of the representatives $r(\sigma)$ must satisfy the following compatibility property.

\begin{definition}
The set of representatives $\{ r(\sigma), \sigma \in \mathcal{O} \}$ is said to be \textit{orbit-compatible} if for any generalized orbit $\mathbb{O}= \cup_{i=1}^{p} \sigma_i$,  $\sigma_1$, ..., $\sigma_{p} \in \mathcal{O}$, there exists 
$ j \mbox{ such that } r(\sigma_j) = r(\sigma_i) \; \mbox{for all $i$ such that } r(\sigma_j) \in \sigma_i
$.
Such a solution $r(\sigma_j)$ is said to be a \textit{generalized representative} of $\mathbb{O}$.
\end{definition}

Note that there always exists a set of orbit-compatible representatives: start by choosing a representative $r(\sigma)$ for a given $\sigma \in \mathcal{O}$, and then choose $r(\sigma)$ as the representative of each orbit $\sigma'$ in which $r(\sigma)$ is contained.
Representatives of orbits not containing $r(\sigma)$ can be chosen arbitrarily.

There may exist several generalized representatives of a given generalized orbit.
If  $\{ r(\sigma), \sigma \in \mathcal{O} \}$ is orbit-compatible then for each generalized orbit $\mathbb{O}= \cup_{i=1}^{p} \sigma_i$ there exists $i \in \{1, ..., p\}$ such that either $r(\sigma_i)$ is not contained in any other orbit $\sigma_j \in \mathcal{O} $, $j \not= i$, or  $r(\sigma_i)$ is the representative of any orbit to which it belongs.
The next lemma states that when representatives are orbit-compatible, there remains at least one element per generalized orbit even if all elements $\cup_{\sigma \in \mathcal{O}}( \sigma \backslash r(\sigma))$ have been removed.

\begin{lemma}
For given orbit-compatible representatives $r(\sigma)$, $\sigma \in \mathcal{O}$, for any generalized orbit $\mathbb{O}= \cup_{i=1}^{p} \sigma_i$, $\sigma_1$, ..., $\sigma_{p} \in \mathcal{O}$,
$ \displaystyle{
\exists j \in \{1, ..., p\} \mbox{ such that } r(\sigma_j) \not\in \cup_{i=1}^p (\sigma_i \backslash r(\sigma_i))}
$.

\end{lemma}

Note that even if the set of representatives is orbit-compatible, it may happen that an entire orbit $\sigma \in \mathcal{O}$ is removed by a sub-symmetry breaking technique. However, if orbit-compatibility is satisfied, there will always remain at least one element in the corresponding generalized orbit, with same cost as any solution in orbit $\sigma$.

Referring to Example \ref{ex_gen_orbit}, we focus on generalized orbit $O$.
In Figure \ref{fig:fig1}, $X_1$ (resp. $X_4$, $X_5$) is chosen to be the representative of orbit $Q_1$ (resp. $Q_3$, $Q_2$).
The set of chosen representatives is not orbit-compatible.
Indeed, there is no generalized representative as each representative belongs also to another orbit, of which it is not representative.
Thus the set of removed elements $\cup_{i=1}^p (\sigma_i \backslash r(\sigma_i))$ contains all elements of the generalized orbit.
In Figure \ref{fig:fig2}, $X_3$ (resp. $X_5$) is chosen to be representative of $Q_1$  (resp. orbits $Q_3$ and $Q_2$). 
In this case, 
the set of chosen representatives is orbit-compatible, since solutions $X_3$ and $X_5$ are generalized representatives of $O$.
Indeed, $X_3$ is representative of $Q_1$ and does not belong to any other orbit, so it remains after removal removal of $\cup_{i=1}^p (\sigma_i \backslash r(\sigma_i))$. 
Solution $X_5$ is representative of $Q_2$ and $Q_3$, and belongs to these two orbits only, so it remains after removal as well.
In Figure \ref{fig:fig3},  $X_1$ (resp. $X_5$) is chosen to be representative of $Q_1$  (resp. orbits $Q_3$ and $Q_2$).
In this case, 
the set of chosen representatives is orbit-compatible, since solution $X_5$ is a generalized representative of $O$.
Indeed, $X_5$ is representative of $Q_2$ and $Q_3$ and does not belong to any other orbit. % so it remains after removal removal of $\cup_{i=1}^p (\sigma_i \backslash r(\sigma_i))$. 
This choice of representatives is certainly the best as there is exactly one generalized representative of $O$.
Indeed, $X_1$ is representative of $Q_1$, but also belongs to orbit $Q_2$ which has another representative. Therefore, $X_1$ is in the set of removed elements $\cup_{i=1}^p (\sigma_i \backslash r(\sigma_i))$. 
%Then only one solution (namely $X_5$) remains, which makes this set of representative also orbit-compatible.

\begin{figure}[ht]
\scriptsize
  \centering

  \begin{subfigure}[b]{0.4\linewidth}
      \begin{tikzpicture}[scale=0.8] 
  
\node (l1) at (0,0) [] {$X_1 = r(Q_1)$};
\node (l0) at (-0.2,0.2) [] {};

\node (c1) at (-0.5,0.35) [] {};
\node (c2) at (4.5,0.35) [] {};

\node (l2) at (1,0.7) [] {$X_2$};
\node (l3) at (3,0.7) [] {$X_3$};

\node (l4) at (4,0) [] {$X_4 = r(Q_3)$};
\node (l8) at (4,0.2) [] {};

\node (l5) at (2,-1.1) [] {$X_5 = r(Q_2) $};
\node (l6) at (3.1, -1.8) []  {};
\node (l7) at (1, -1.8) []  {};

\node (q1) at (-0.7, 1.3) []  {$Q_1$};
\node (q3) at (4.5, -1.2) []  {$Q_3$};
\node (q2) at (-0.5, -1.2) []  {$Q_2$};

\draw let \p1=(c1), \p2=(c2), \n1={atan2(\y2-\y1,\x2-\x1)}, \n2={veclen(\y2-\y1,\x2-\x1)}
  in ($ (c2)!0.5!(c1) $) ellipse [x radius=\n2/2+20pt, y radius=1cm,rotate=180-\n1];
  
  \draw[dashed] let \p1=(l0), \p2=(l6), \n1={atan2(\y2-\y1,\x2-\x1)}, \n2={veclen(\y2-\y1,\x2-\x1)}
  in ($ (l0)!0.42!(l6) $) ellipse [x radius=\n2/2+16pt, y radius=0.8cm,rotate=132-\n1];
  
    \draw[densely dotted] let \p1=(l8), \p2=(l7), \n1={atan2(\y2-\y1,\x2-\x1)}, \n2={veclen(\y2-\y1,\x2-\x1)}
  in ($ (l8)!0.42!(l7) $) ellipse [x radius=\n2/2+18pt, y radius=0.8cm,rotate=228-\n1];

\end{tikzpicture}
  \caption{\label{fig:fig1} No element remaining.}
\end{subfigure}%

  \begin{subfigure}[b]{0.4\linewidth}
        \begin{tikzpicture}[scale=0.8] 
\node (l1) at (0,0) [] {$X_1$};
\node (l0) at (-0.2,0.2) [] {};

\node (c1) at (-0.5,0.35) [] {};
\node (c2) at (4.5,0.35) [] {};

\node (l2) at (1,0.7) [] {$X_2$};
\node (l3) at (3,0.7) [] {$X_3 = r(Q_1)$};

\node (l4) at (4,0) [] {$X_4$};
\node (l8) at (4,0.2) [] {};

\node (l5) at (2,-1.1) [] {$X_5 = r(Q_2) = r(Q_3)$};
\node (l6) at (3.1, -1.8) []  {};
\node (l7) at (1, -1.8) []  {};

\node (q1) at (-0.7, 1.3) []  {$Q_1$};
\node (q3) at (4.5, -1.2) []  {$Q_3$};
\node (q2) at (-0.5, -1.2) []  {$Q_2$};

\draw let \p1=(c1), \p2=(c2), \n1={atan2(\y2-\y1,\x2-\x1)}, \n2={veclen(\y2-\y1,\x2-\x1)}
  in ($ (c2)!0.5!(c1) $) ellipse [x radius=\n2/2+20pt, y radius=1cm,rotate=180-\n1];
  
  \draw[dashed] let \p1=(l0), \p2=(l6), \n1={atan2(\y2-\y1,\x2-\x1)}, \n2={veclen(\y2-\y1,\x2-\x1)}
  in ($ (l0)!0.42!(l6) $) ellipse [x radius=\n2/2+16pt, y radius=0.8cm,rotate=132-\n1];
  
    \draw[densely dotted] let \p1=(l8), \p2=(l7), \n1={atan2(\y2-\y1,\x2-\x1)}, \n2={veclen(\y2-\y1,\x2-\x1)}
  in ($ (l8)!0.42!(l7) $) ellipse [x radius=\n2/2+18pt, y radius=0.8cm,rotate=228-\n1];
\end{tikzpicture}
\caption{\label{fig:fig2} $X_3$ and $X_5$ remain}
\end{subfigure}%
  \begin{subfigure}[b]{0.3\linewidth}
      \begin{tikzpicture}[scale=0.8] 
  
\node (l1) at (0,0) [] {$X_1 = r(Q_1)$};
\node (l0) at (-0.2,0.2) [] {};

\node (c1) at (-0.5,0.35) [] {};
\node (c2) at (4.5,0.35) [] {};

\node (l2) at (1,0.7) [] {$X_2$};
\node (l3) at (3,0.7) [] {$X_3$};

\node (l4) at (4,0) [] {$X_4$};
\node (l8) at (4,0.2) [] {};

\node (l5) at (2,-1.1) [] {$X_5 = r(Q_2) = r(Q_3)$};
\node (l6) at (3.1, -1.8) []  {};
\node (l7) at (1, -1.8) []  {};

\node (q1) at (-0.7, 1.3) []  {$Q_1$};
\node (q3) at (4.5, -1.2) []  {$Q_3$};
\node (q2) at (-0.5, -1.2) []  {$Q_2$};

\draw let \p1=(c1), \p2=(c2), \n1={atan2(\y2-\y1,\x2-\x1)}, \n2={veclen(\y2-\y1,\x2-\x1)}
  in ($ (c2)!0.5!(c1) $) ellipse [x radius=\n2/2+20pt, y radius=1cm,rotate=180-\n1];
  
  \draw[dashed] let \p1=(l0), \p2=(l6), \n1={atan2(\y2-\y1,\x2-\x1)}, \n2={veclen(\y2-\y1,\x2-\x1)}
  in ($ (l0)!0.42!(l6) $) ellipse [x radius=\n2/2+16pt, y radius=0.8cm,rotate=132-\n1];
  
    \draw[densely dotted] let \p1=(l8), \p2=(l7), \n1={atan2(\y2-\y1,\x2-\x1)}, \n2={veclen(\y2-\y1,\x2-\x1)}
  in ($ (l8)!0.42!(l7) $) ellipse [x radius=\n2/2+18pt, y radius=0.8cm,rotate=228-\n1];

\end{tikzpicture}
  \caption{\label{fig:fig3} $X_5$ remains}
\end{subfigure}%
  \caption{\label{fig:orbit}Orbits in the generalized orbit $\{X_1, X_2, X_3, X_4, X_5 \}$  with various choices of representatives}
  \label{diagram}
  
  \end{figure}
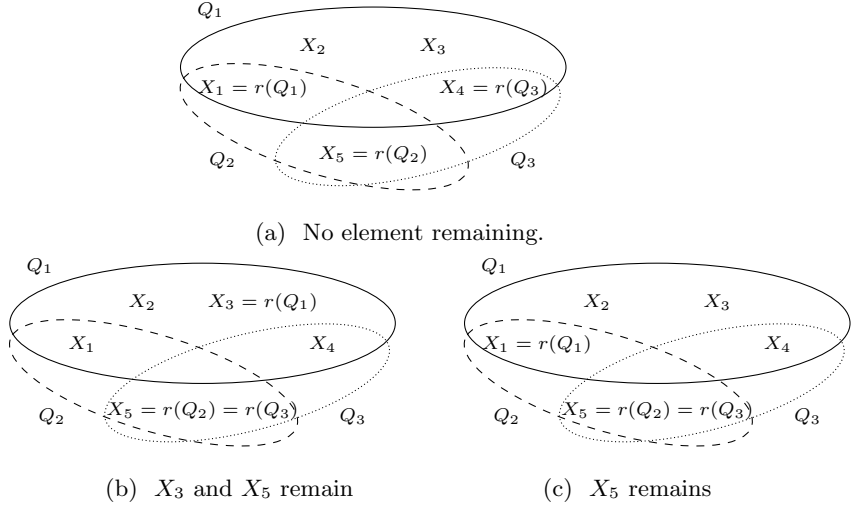

\subsection{Full sub-orbitopes}

Given $X \in \mathcal{X}$ and sets $R \subset \{1, ..., m \}$ and $C \subset \{1, ..., n\}$, we consider sub-matrix $(R,C)$ of $X$, denoted by $X(R,C)$, obtained by considering columns $C$ of $X$ on rows $R$ only.
Symmetry group $\mathcal{G}_{Q}$ is the \textit{sub-symmetric group} with respect to $(R, C)$ if it is the set of all permutations of 
the columns of $X(R, C)$.

In this section, we generalize the notion of full orbitope in order to account for sub-symmetries arising in subproblems of ILP (\ref{ILP_orb}) whose symmetry group is a sub-symmetric group.
We consider solutions subsets ${Q_i}$, $i \in \{1, ..., s\}$, such that for each $i$ the symmetry group $\mathcal{G}_{Q_i}$ is the sub-symmetric group with respect to $(R_i, C_i)$, $R_i \subset \{1, ..., m \}$ and $C_i \subset \{1, ..., n\}$.

For each orbit $O^i_k$, $k \in \{1, ..., o_i\}$, let its representative $X^i_k \in O^i_k$ be such that sub-matrix $X^i_k(R_i, C_i)$ is lexicographically maximal, $i.e.$, its columns are lexicographically non-increasing.

\begin{lemma}
\label{compatib}

The set of representatives $\{ X^i_k, \; k \in \{1, ..., o_i\}, \; i \in \{1, ..., s\} \}$ is orbit-compatible.
%This set of representatives is orbit-compatible., and for any generalized orbit representative $X$ and $i \in \{1, ..., s\}$, if $X \in Q_i$ then sub-matrix $(R_i, C_i)$ of $X$ is lexicographically maximal, $i.e.$, its columns are lexicographically non-increasing.
\end{lemma}

\begin{proof}
In order to prove that this set of representatives is orbit-compatible, we prove that there exists a generalized representative of each generalized orbit.

 First consider the following row-wise ordering of matrix entries:
$
(1,1), \; (1,2), \; ...,  \; (1,n)$, $(2,1),$ $ \;(2,2), $ $ \; ...$, $(2,n),  \;...,  \;(m,n)
$.
We define an ordering $\succ_M$ of the matrices such that for two matrices $A$ and $B$, $A \succ_M B$ if $A(i,j) > B(i,j)$, with $(i,j)$ the first position, with respect to the given ordering of matrix entries, where $A$ and $B$ differ.
%Note that a matrix $X$ is maximal with respect to ordering $\succ$ if its columns are lexicographically non-increasing.

For a given solution matrix $X \in \mathcal{X}$, 
we propose an algorithm computing a generalized representative of the generalized orbit of $X$. First set $X^0 = X$.
At iteration $k$ of the algorithm, there are two cases. In the first case, there exists $i \in \{1, ..., s\}$ such that $X^k \in Q_i$ and sub-matrix $X^k(R_i, C_i)$ is not lexicographically maximal, $i.e.$, there exists a column $j \in C_i$ such that $X^k(R_i, \{j\}) \prec X^k(R_i, \{j+1\})$. In this case, $X^{k+1}$ is set to $X^k$, except that columns $j$ and $j+1$ of sub-matrix $X^k(R_i, C_i)$ are transposed. Otherwise in the second case, the algorithm stops.
The claim is that this algorithm stops at some iteration $K$, and corresponding matrix $X^K$ is a generalized representative of the generalized orbit of $X$.
Note that at each iteration $k$,  $X^{k} \succ_M X^{k-1}$. 
As matrices $X^{k}$ take values in a finite set, there exists an iteration $K$ at which the algorithm stops.
By construction, 
matrix $X^K$ is in the generalized orbit of $X$, and for each $i \in \{1, ..., s\}$ such that $X \in Q_i$, sub-matrix $(R_i, C_i)$ of $X$ is lexicographically maximal. It is thus a generalized representative of the generalized orbit of $X$.
\qed
\end{proof}

The \textit{full sub-orbitope} $\mathcal{P}_{sub}$ with respect to subsets ${Q_i}$,  $i \in \{1, ..., s\}$, is the convex hull of binary matrices $X$ such that for each $i \in \{1, ..., s\}$, if $X \in Q_i$ then the columns of $X(R_i, C_i)$ are lexicographically non-increasing. In particular, $\mathcal{P}_{sub}$ contains the generalized representatives of each generalized orbit $\mathcal{O}$, but no other element of $\mathcal{O}$.
Note that the full sub-orbitope generalizes the full orbitope, as for $s=1$, $Q_1=\mathcal{X}$, $\mathcal{G}_{Q_1}= \mathfrak{S}_n$ and $(R_1, C_1) = (\{1, ..., m\}, \{1, ..., n\}) $, the associated full sub-orbitope is the full orbitope $\mathcal{P}_0(m,n)$.
Note that the restriction $\mathcal{P}_{sub}(R_i,C_i)$ of $\mathcal{P}_{sub}$ to sub-matrix $(R_i,C_i)$ is the full orbitope $\mathcal{P}_0(|R_i|, |C_i|)$ for any $i\in \{1, ..., s\}$.
The results presented in Section \ref{sec:intersection} can then directly be used to compute the smallest cube face containing the binary points in the intersection of $\mathcal{P}_{sub}(R_i,C_i)$ and a given cube face.

\section{Static and dynamic orbitopal fixing}
\label{sec:OF}

So far, the considered  lexicographical order on the columns was defined with respect to  order $1, ...., m$ on the rows.
In this section, we define a \textit{static orbitopal fixing} algorithm for the full orbitope, which relies on this lexicographical order.
We also define a \textit{dynamic orbitopal fixing} algorithm for the full orbitope, where the lexicographical order is defined with respect to an order on the rows determined by the branching decisions in the B\&B tree.
Interestingly these static and dynamic variants of the orbitopal fixing algorithm can be used also for the full sub-orbitope case.
It is worth noting that this orbitopal fixing algorithm based on our intersection result from Section \ref{sec:intersection} performs all possible variable fixings (with respect to the full (sub-)orbitope) as early as possible in the B\&B tree.

\subsection{Static orbitopal fixing}
\label{static_fix}

When solving MILP (\ref{ILP_orb}) with B\&B, static orbitopal fixing can be performed at the beginning of each node processing in the B\&B tree, in order to ensure that any enumerated solution $x$ in the B\&B tree is such that $x \in \mathcal{P}_0(m,n)$, assuming the lexicographical order is a priori given.

The \textit{static orbitopal fixing} algorithm at node $a$ is the following:
\vspace{-2.5mm}
\begin{itemize}

    \item[-] Set $I_0 = B^a_0$, $I_1 = B^a_1$,   where $B^a_0$ (resp. $B^a_1$) is the index set of variables previously fixed to 0 (resp. 1).
    \item[-] Compute matrices $\underline{{M}}^1$ and $\overline{{M}}^n$  using Algorithm 1. 
    \item[-] If $\underline{{M}}^1$ or $\overline{{M}}^n$ is defined by pair $(S^{\varnothing}_0,S^{\varnothing}_1)$, prune node $a$. Otherwise determine $I^{+}_0$ and $I^{+}_1$  using Th. \ref{I0I1}.    
    \item[-] Fix variable $x_{i,j}$ to 0, for each $(i,j) \in I_0^+$.
    \item[-] Fix variable $x_{i,j}$ to 1, for each $(i,j) \in I_1^+$.
\end{itemize}

From Theorem \ref{I0I1}, the pair $(I^\star_0, I^\star_1)$ = ($I_0 \cup I_0^+$, $I_1 \cup I_1^+$) defines $\mbox{Fix}_{F(a)}(\mathcal{P}_0(m,n))$, where $F(a)$ is the hypercube face given by $(B^a_0, B^a_1)$ at each node $a$ of a B\&B tree.
Thus the following result can be directly derived.

\begin{theorem}
\label{valid_statfix}
Let $\tau$ be a B\&B tree of ILP (\ref{ILP_orb}), in which static orbitopal fixing is performed, and the branching rule is arbitrary.
For each solution orbit $\mathcal{O}$  of ILP (\ref{ILP_orb}), there is exactly one solution of $\mathcal{O}$ enumerated in B\&B tree $\tau$.
\end{theorem}

From Theorem \ref{lineartime},  the static orbitopal fixing algorithm is   in $O(mn)$ time at each node of the B\&B tree.

\subsection{Dynamic orbitopal fixing}
\label{dynamic_fix}

%In the previous sections, we have considered that a column $c$ of an $m \times n$ binary matrix is lexicographically greater than or equal to a column $c'$ if there exists $i \in \{1, ...., m-1\}$ such that $\forall i' \leq i$, $y_{i'} = z_{i'}$ and $y_{i+1} > z_{i+1}$.

In the previous sections, the lexicographical order on the columns  of an $m \times n$ binary matrix was defined with respect to the order $1, ..., m$ on the rows.
Note that this order is arbitrary, and thus the definition of the lexicographical order can be extended for any ordering of the $m$ rows. Namely, considering a bijection $\phi : \{1, ...,m \} \rightarrow \{1, ..., m\}$, column $c$ is lexicographically greater than or equal to a column $c'$, with respect to ordering $\phi$, if there exists $i \in \{1, ...., m-1\}$ such that
 $\forall i' \leq i$, $y_{\phi(i')} = z_{\phi(i')}$
and $y_{\phi(i)+1} > z_{\phi(i)+1}$.

Dynamic fixing is to perform, at any node $a$ of the B\&B tree, orbitopal fixing with respect to reorderings $\phi_a$ of the row indices, defined by the branching decisions leading to node $a$.
The idea of pruning the B\&B tree with respect to an order defined by the branching process has been introduced by Margot \cite{Margot10}.

As a first step,  suppose at each node $a$ of the B\&B tree, the branching disjunction has the form
\begin{equation}
\label{disj}
x_{i_a,j_a} = 0 \quad \vee \quad x_{i_a,j_a} = 1.
\end{equation}

\textit{Dynamic orbitopal fixing} is to perform orbitopal fixing  on row set $I_a = \{$ $\phi_a(1)$, $\phi_a(2)$, ..., $\phi_a(|I_a|)$ $\}$  at each node $a$, where lexicographical ordering $\phi_a$ and $I_a$ are defined   recursively   as follows.
$$\mbox{If $a$ is the root, then } 
\left\{
\begin{array}{ll}
   I_a = \{ i_a \}\\
   \phi_a(1) = i_a \\
\end{array}
\right.
\mbox{, otherwise  }
\left\{
\begin{array}{ll}
    I_a = I_b \cup \{ i_a \}\\
   \phi_a(i) = \phi_b(i) &  \forall i \in \{ 1, ..., |I_b| \}.\\
    \phi_a(|I_b|+1) = i_a & \mbox{if } i_a \not\in I_b, \\
    \mbox{where $b$ is the father of $a$.} \\
\end{array}
\right.
$$

%then for each node $a$ of the B\&B tree, we consider the $d_a - 1$ branching decisions leading from the root node to node $a$:$$x^{i_1}_{t_1} = v_1, \; x^{i_2}_{t_2} = v_2, \; ..., \; x^{i^{d_a - 1}}_{t_{d_a - 1}} = v_{d_a - 1}.$$ Consider also the branching disjunction at node $a$:$$x^{i^{d_a}}_{t_{d_a}} = 0 \quad \vee \quad x^{i^{d_a}}_{t_{d_a}} = 1,$$Consider sequence $(t_i)_{(i \in \{ 1, ..., d_a \})}$. For each $i \in \{ 1, ..., d_a \}$, time $t_i$ is removed from the sequence if it has already appeared before: $\exists j < i$ such that $t_j = t_i$.Once all redundant elements are removed, sequence $(t_i)_{(i \in \{ 1, ..., d_a \})}$ can be reindexed from $1$ to $d_a'$, $d_a' \leq d_a$: $(t_i)_{(i \in \{ 1, ..., d_a' \})}$. At node $a$, ordering $\phi_a$ is then partially defined as follows:$$\phi_a(1) = t_1, \; ..., \; \phi_a(d_a') = t_{d_a'}$$

Note that an arbitrary branching rule used alongside with an arbitrary ordering may lead to the removal of every optimal solution from the B\&B tree. The following theorem shows that the use of branching rule (\ref{disj}) and ordering $\phi_a$ preserves an optimal solution in the B\&B tree.

\begin{theorem}
\label{valid_dynfix}
Let $\tau$ be a B\&B tree of ILP (\ref{ILP_orb}), in which dynamic orbitopal fixing is performed and branching disjunctions have the form (\ref{disj}).
For each solution orbit $\mathcal{O}$ of ILP (\ref{ILP_orb}), there is exactly one solution of $\mathcal{O}$ enumerated in B\&B tree $\tau$.
\end{theorem}

\begin{proof}
  The sketch of the proof is to produce a solution $X \in \mathcal{O}$ and prove that $X$ is the only solution of $\mathcal{O}$ which is enumerated in $\tau$.

%The idea of the proof is to compute the solution $X' \in \mathcal{O}$ enumerated in $\tau$ by following the branching decisions taken in $\tau$. This will also prove there is a unique solution of $\mathcal{O}$ enumerated in $\tau$.
First consider the branching disjunction at the root node $a_r$:
$
(x_{i_0,j_0} =  0 \quad \vee \quad x_{i_0,j_0} = 1).
$
Then $\phi_{a_r}(1) = i_0$.
Let $n_{i_0}$ be the number of 1-entries on row $i_0$ of any matrix $X \in \mathcal{O}$.
 Since  dynamic orbitopal fixing is enforced in $\tau$, any solution enumerated by $\tau$ must be lexicographically non-increasing with respect to $\phi_{a_r}$.
Then, as row $i_0$ is the first row with respect to the lexicographical order $\phi_{a_r}$,  any $X \in  \mathcal{O}$ enumerated by the B\&B tree will be such that:
$$
X(i_0, j) = 1, \; \forall j \in \{1, ..., n_{i_0}\}
\qquad \mbox{and} \qquad 
X(i_0, j) = 0, \; \forall j \in \{n_{i_0}+1, ..., n\}
$$

%Thus, if we denote by $b_1$ the son of $a_r$ such that $x_{i_0,j_0} =  1$ and $b_0$ the son of $a$ such that $x_{i_0,j_0} = 0$, then the next node to be considered will be $b_1$ if $j_0 \leq n_{i_0}$ and $b_0$ otherwise.

Note that if $j_0 \leq n_{i_0}$ (resp. $j_0 > n_{i_0}$) then any $X \in  \mathcal{O}$ enumerated by $\tau$ is such that $X(i_0, j_0) = 1$ (resp. $X(i_0, j_0) = 0$). Thus there is no solution of $\mathcal{O}$ in the branch $x_{i_0,j_0} =  0$ (resp. $x_{i_0,j_0} = 1 $).

 Suppose  w.l.o.g. that $j_0 \leq n_{i_0}$, so the node considered is $b_1$, the son of $a_r$ such that $x_{i_0,j_0} =  1$.
Consider the branching disjunction at node $b_1$:
$(x_{i_1,j_1} =  0 \quad \vee \quad x_{i_1,j_1} = 1)$.
If $i_1 = i_0$ then, by the same arguments as at the root node, there is exactly one branch in which all elements of $ \mathcal{O}$ are enumerated, and this branch can be easily determined.
Otherwise, since $i_1 \not= i_0$, then by construction, $\phi_{b_1}(1)=i_0$ and $\phi_{b_1}(2)=i_1$.
Let $n^1_{i_1}$ (resp. $n^0_{i_1}$) be the number of columns $j$ such that $X(i_0,j)=1$ (resp. $X(i_0,j)=0$) and $X(i_1,j)=1$.
Since row $i_1$ is second with respect to lexicographical order $\phi_{b_1}$, any $X\in  \mathcal{O}$ enumerated by the B\&B tree will be such that:
$$
\left\{
\begin{array}{ll}
   X(i_1, j) = 1 \quad  & \forall j \in \{1, ..., n^1_{i_1}\} \cup \{n_{i_0} + 1, ..., n_{i_0}+ n^0_{i_1}\}\\
   X(i_1, j) = 0 &  \forall j \in \{n^1_{i_1} + 1, ..., n_{i_0} \} \cup  \{n_{i_0} + n^0_{i_1} + 1, ..., n \} \\
\end{array}
\right.
$$

Thus, all $X \in  \mathcal{O}$ enumerated by $\tau$ have the same value $v$ in entry $(i_1,j_1)$, and this value can be determined, as previously, by finding to which of the sets $\{1, ..., n^1_{i_1}\}$, $\{n^1_{i_1} + 1, ..., n_{i_0} \}$, $\{n_{i_0} + 1, ..., n_{i_0}+ n^0_{i_1}\}$, $\{n_{i_0}+ n^0_{i_1} + 1, ..., n \}$ does index $j_1$ belong.
Therefore, since for all $X \in  \mathcal{O}$ enumerated by $\tau$, $X(i_1,j_1) = v$, there is exactly one branch ($x_{i_1,j_1} =  v$) in which any $X \in  \mathcal{O}$ is enumerated. This process can be repeated until a leaf node $a_l$ is reached. At that point, all entries of $X$ are determined. By construction, $X$ is the only element of $ \mathcal{O}$ enumerated by $\tau$, since at each node we considered, there was always a unique branch leading to all elements of $ \mathcal{O}$.
 
\qed \end{proof}

Now suppose the branching disjunction at each node $a$ is arbritrary.
The latter result can be extended to show that dynamic orbitopal fixing can also be used in this case.
For each node $a$, consider a branching disjunction of the form:
\begin{equation}
\label{disj_arb}
\sum_{i \in \mathcal{R}_a} \sum_{j = 1}^p  \lambda^i_a x_{i,j} \leq k \quad \vee \quad \sum_{i \in \mathcal{R}_a} \sum_{j = 1}^p \lambda^i_a x_{i,j} > k.
\end{equation}
where $\mathcal{R}_a = \{ r_{a,1}, ..., r_{a,p} \} \subset \{1, ..., m\}$.

A new lexicographical ordering $\widetilde{\phi}_a$ taking into account every row involved in disjunction (\ref{disj_arb}) must be defined at each node $a$. 
Namely, row subset $\widetilde{I}_a \subset \{ 1, ..., m\}$ and bijection $\widetilde{\phi}_a : \{ 1, ..., |\widetilde{I}_a|  \} \rightarrow \widetilde{I}_a$ are as follows.

$$ \mbox{If $a$ is the root, then }
\left\{
\begin{array}{ll}
   \widetilde{I}_a = \mathcal{R}_a &\\
   \widetilde{\phi}_a(k) = r_{a,k}, \; & k \in \{ 1, ..., p \} \\
\end{array}
\right.
\mbox{, otherwise }
\left\{
\begin{array}{ll}
   \widetilde{I}_a = \widetilde{I}_b \cup \mathcal{R}_a &\\
   \widetilde{\phi}_a(i) = \widetilde{\phi}_b(i) &  \forall i \in \{ 1, ..., |\widetilde{I}_b| \}\\
   \widetilde{\phi}_a(|\widetilde{I}_b|+k) = r'_{a,k}, & \forall k \in \{ 1, ..., p' \}  \\
   \mbox{where $b$ is the father of $a$} \\
   \mbox{and $ \{ r'_{a,1}, ..., r'_{a,p'} \}  = \mathcal{R}_a \backslash \widetilde{I}_b $}.
\end{array}
\right.
$$

\subsection{Orbitopal fixing in the full sub-orbitope}

Recall that given $X \in \mathcal{X}$ and sets $R \subset \{1, ..., m \}$ and $C \subset \{1, ..., n\}$, sub-matrix $X(R,C)$ of $X$ is obtained by considering columns $C$ of $X$ on rows $R$ only.
Consider solutions subsets ${Q_i} \subset \mathcal{X}$, $i \in \{1, ..., s\}$ such that for each $i \in \{1, ..., s\}$ the symmetry group $\mathcal{G}_{Q_i}$ is the sub-symmetric group with respect to $(R_i, C_i)$. Let $\mathcal{P}_{sub}$ be the associated full sub-orbitope.
Static (resp. dynamic) orbitopal fixing can be performed for $\mathcal{P}_{sub}$ at each node $a$ of the B\&B tree as follows.
Consider $\mathfrak{I}_a \subset \{1, ..., s\}$ such that
for each $i \in \mathfrak{I}_a$, each solution $x$ to the sub-problem at node $a$ is in $Q_i$.
The idea is to apply static (resp. dynamic) orbitopal fixing to the sub-matrix $X(R_i, C_i)$, for each $i \in \mathfrak{I}_a$.

By Lemma \ref{compatib} the representatives associated with the natural lexicographical order are orbit-compatible.
%, and any generalized orbit representative $X$ is such that any sub-matrix $X(R_i, C_i)$ is lexicographically maximal.
Consequently, static orbitopal fixing for $\mathcal{P}_{sub}$ does not change the optimal value returned by the B\&B process.
Lemma \ref{compatib} can directly be adapted to the case when the representatives are associated to a lexicographical order defined by arbitrary row-order $\phi$, and the proof of Theorem \ref{valid_dynfix} can be slightly modified to show that dynamic orbitopal fixing for $\mathcal{P}_{sub}$ is also valid.

\section{Application to the Unit Commitment Problem}
\label{sec:MUCP}

Given a discrete time horizon $\mathcal{T} = \{1, ...,  T \}$, a demand for electric power $D_t$ is to be met at each time period $t \in \mathcal{T}$. Power is provided by a set $\mathcal{N}$ of $n$ production units. 
At each time period, unit $j \in \mathcal{N}$ is either down or up, and in the latter case, its production is within [$P^j_{min}$, $P^j_{max}$].
Each unit must satisfy minimum up-time (resp. down-time) constraints, \textit{i.e.}, it must remain up (resp. down) during at least $L^j$ (resp. $\ell^j$) periods after start up (resp. shut down). 
Each unit $j$ also features three different costs: a fixed cost $c^j_f$, incurred each time period the unit is up; a start-up cost $c_0^j$, incurred each time the unit starts up; and a cost $c^j_p$ proportional to its production.
The Min-up/min-down Unit Commitment Problem (MUCP) is to find a production plan minimizing the total cost while satisfying the demand and the minimum up and down time constraints. 
The MUCP is strongly NP-hard \cite{ComplexityMUCP}.

For each unit $j \in \mathcal{N}$ and time period $t \in \mathcal{T}$, let us consider three variables: $x_{t,j} \in \{ 0,1 \}$ indicates if unit $j$ is up at time $t$; $u_{t,j} \in \{ 0,1 \}$ whether unit $j$ starts up at time $t$; and $p_{t,j} \in \mathbb{R}$ is the quantity of power produced by unit $j$ at time $t$.
Without loss of generality we consider that $L^j, \ell^j \leq T$.
A formulation for the MUCP is as follows \cite{Takriti05,Ostrowski15,MUCP16}.

\begin{alignat}{3}
& \underset{x,u,p}{\text{min}}
& & \sum_{j=1}^{n} \sum_{t=1}^{T} c^j_f x_{t,j} + c^j_p p_{t,j} + c^j_0 u_{t,j} \nonumber \\
& \text{s. t.} \quad
& & \sum_{t' = t - L^j + 1}^{t} u_{t',j} \leq x_{t,j} & \forall j \in \mathcal{N},  \; \forall t \in \{ L^j, ..., T\} \label{min-up} \\
&&& \sum_{t' = t - \ell^j + 1}^{t} u_{t',j} \leq 1 - x_{t-\ell^j, j} & \forall j \in \mathcal{N},  \; \forall t \in \{ \ell^j, ..., T\}  \label{min-down}\\
&&& u_{t,j} \geq x_{t,j} - x_{t-1,j} & \forall j \in \mathcal{N},  \; \forall t \in \{2, ..., T \} \label{logical} \\
&&& \sum_{j=1}^{n} p_{t,j} \geq D_t & \forall t \in \mathcal{T} \label{demande} \\
&&& P_{min}0^j x_{t,j} \leq p_{t,j} \leq P_{max}^j x_{t,j} & \forall j \in \mathcal{N},  \; \forall t \in \mathcal{T} \label{pow}\\
&&& x_{t,j}, u_{t,j} \in \{ 0,1 \} & \forall j \in \mathcal{N}, \; \forall t \in \mathcal{T} \label{integrity}
\end{alignat}

\subsection{Symmetries in the MUCP}

Symmetries in the MUCP (and in the UCP) arise from the existence of groups of identical units, $i.e.$, units with identical characteristics ($P_{min}$, $P_{max}$, $L$, $\ell$, $c_f$, $c_0$, $c_p$). 
The instance is partitioned into \textit{types} $h \in \{1, ..., H\}$ of $n_h$ identical units.
%Suppose there are $H$ different types of units and $n_h$ units of type $h$, $h \in \{1, ..., H \}$. 
The unit set of type $h$ is denoted by $\mathcal{N}_h$.

The solutions of the MUCP can be expressed as a series of binary matrices.
For a given type $h$, we introduce matrix $X^h \in \mathcal{P}(T, n_h)$ such that entry $X^h(t,j)$ corresponds to variable $x_{t,j'}$, where $j'$ is the index of the $j^{th}$ unit of type $h$. Column of matrix $X^h(j)$ corresponds to the up/down plan relative to the $j^{th}$ unit of type $h$. Similarly, we introduce matrices $U^h$ and $P^h$.

The set of all feasible $X = (x_{t,j})_{t \in \mathcal{T}, j \in \mathcal{N}}$ is denoted by $\mathcal{X}_{MUCP}$.
Note that any solution matrix $X$ (resp. $U$, $P$) can be partitioned in $H$ matrices $X^h$ (resp. $U^h$, $P^h$).
Since all units of type $h$ are identical, their production plans can be permuted, provided that the same permutation is applied to matrices $X^h$, $U^h$ and $P^h$.
Thus, the symmetry group $\mathcal{G}$ contains all column permutations applied to $X^h$, $U^h$ and $P^h$ for each unit type $h$.
Consequently, for each type $h$, feasible solutions $X^h$ can be restricted to be in the full orbitope $\mathcal{P}_0(T, n_h)$.
As binary variables $U$ are uniquely determined by variables $X$, breaking the symmetry on the $X$ variables will break the symmetry on $U$ variables.

\subsection{Sub-symmetries in the MUCP}
\label{sub-sym-MUCP}

There are other sources of symmetry, arising from the possibility, in some cases, of permuting some sub-columns of matrices $X^h$.
For example, consider two identical units at a given node $a$. Suppose the fixings at the previous nodes are such that these two units are down and ready to start up at given time $t_0$. Then their plans after $t_0$ can be permuted, even if they do not have the same up/down plan before $t_0$.
This kind of sub-symmetry is not detected by the symmetry group $\mathcal{G}$. Indeed, as soon as their up/down plans before $t_0$ are different, the two units would no longer be considered symmetrical with respect to $\mathcal{G}$.

More precisely, a unit $j \in \mathcal{N}$ is \textit{ready to start up} at time $t$ if and only if $\forall t' \in \{ t - \ell^j, ..., t - 1\}$, $x_{t', j} = 0$.
Similarly, a unit $j \in \mathcal{N}^k$ is \textit{ready to shut down} at time $t$ if and only if $\forall t' \in \{ t - L^j, ..., t - 1\}$, $x_{t', j} = 1$.

For each time period $ t \in \{1, ...,\mathcal{T}\}$ and subset $ \mathfrak{N} \subset \mathcal{N}_h, \; h \in \{1, ..., H\} $ of identical units, consider the following subsets of $\mathcal{X}_{MUCP}$:

$$
\overline{Q}^t_\mathfrak{N} = \big\{ X \in \mathcal{X}_{MUCP} \; | \; X(t',j) = 0, \; \forall t' \in \{ t - \ell^j, ..., t - 1\}, \; \forall j \in \mathfrak{N} \big\}
$$
$$
\underline{Q}^t_\mathfrak{N} = \big\{  X \in \mathcal{X}_{MUCP} \; | \; X(t',j) = 1, \; \forall t' \in \{ t - L^j, ..., t - 1\}, \; \forall i \in \mathfrak{N} \big\}
$$

Note that at each node $a$ of the tree, it is easy to find the sets $\overline{Q}^t_\mathfrak{N}$ and $\underline{Q}^t_\mathfrak{N}$, $ t \in \{1, ...,\mathcal{T}\}$, $\mathfrak{N} \subset \mathcal{N}_h, \; h \in \{1, ..., H\}$, to which any solution of the subproblem associated to node $a$ belongs.
Indeed, for each time period $t$ and for each unit $i$ down (resp. up) at time $t$, it is possible to know how long unit $i$ has been down (resp. up), and thus whether unit $i$ is ready to start up (resp. shut down) or not. 
If we denote by $\mathfrak{N}^u_{t,h}$ (resp. $\mathfrak{N}^d_{t,h}$) the set of type $h$ units which are ready to start up (resp. shut down) at time $t$, then all solutions at node $a$ are in sets $\overline{Q}^t_{\mathfrak{N}^u_{t,h}}$ and $\underline{Q}^t_{\mathfrak{N}^d_{t,h}}$.

Let $\mathcal{G}_{\overline{Q}^t_\mathfrak{N}}$ and $\mathcal{G}_{\underline{Q}^t_\mathfrak{N}}$ be the sub-symmetry groups associated to  $\overline{Q}^t_\mathfrak{N}$ and $\underline{Q}^t_\mathfrak{N}$, $ t \in \{1, ...,\mathcal{T}\}$, $\mathfrak{N} \subset \mathcal{N}_h, \; h \in \{1, ..., H\}$.
Note that groups $\mathcal{G}_{\overline{Q}^t_\mathfrak{N}}$ and $\mathcal{G}_{\underline{Q}^t_\mathfrak{N}}$ contain the sub-symmetric group associated to the sub-matrix defined by rows and columns $(\{t, ..., T\}, \mathfrak{N})$.
The corresponding full sub-orbitope is denoted by $\mathcal{P}_{sub}(MUCP)$.

\subsection{Orbitopal fixing for the MUCP}

As the production plans of identical units can be permuted, each variable matrix $X^h$ can be restricted to be in the full orbitope $\mathcal{P}_0(T, n_h)$.
More generally we have seen in Section \ref{sub-sym-MUCP} that variable matrix $X$ can be restricted to be in the full sub-orbitope $\mathcal{P}_{sub}(MUCP)$
 
The fixing strategies developed in Sections \ref{static_fix} and \ref{dynamic_fix} can thus be applied to fix variables in each matrix $X^h$, in order to enumerate only solutions with lexicographically maximal $X^h$.  These strategies can also be applied to restrict the feasible set to the full sub-orbitope $\mathcal{P}_{sub}(MUCP)$.

The possible approaches are the following:
\begin{itemize}
\item[-] Static orbitopal fixing (SOF) for the full orbitopes $\mathcal{P}_0(T, n_h)$, $h \in \{1, ..., H \}$, where the order on the rows is decided before the branching process. 
\item[-] Dynamic orbitopal fixing (DOF) for the full orbitopes $\mathcal{P}_0(T, n_h)$, $h \in \{1, ..., H \}$, where the order on the rows $\widetilde{\phi}$ is decided during the branching process, as described in Section \ref{dynamic_fix}.
\item[-] Static orbitopal fixing for the full orbitopes $\mathcal{P}_0(T, n_h)$, $h \in \{1, ..., H \}$ and for the full sub-orbitope $\mathcal{P}_{sub}(MUCP)$.
\item[-] Dynamic orbitopal fixing for the full orbitopes $\mathcal{P}_0(T, n_h)$, $h \in \{1, ..., H \}$ and for the full sub-orbitope $\mathcal{P}_{sub}(MUCP)$.
\end{itemize}

In the static case, the branching decisions are completely free. As stated in Section \ref{dynamic_fix}, the branching decisions remain free in the dynamic case, provided that the corresponding rows are ordered accordingly.
In our experiments, we only consider the branching disjunctions of the form $(x_{t,j} =  0 \; \vee \; x_{t,j} = 1)$, or $
(x_{t,j} - x_{t-1,j} \leq 0 \quad \vee \quad x_{t,j} - x_{t-1,j } = 1)
$, $i.e.$, $(u_{t,j} =  0 \; \vee \; u_{t,j} = 1)$.

\section{Experimental results for the MUCP}
\label{exp_orbitope}

All experiments were performed using one thread of a PC with a 64 bit Intel Core i7-6700 processor running at 3.4GHz, and 32 GB of RAM memory. The MUCP instances are solved until optimality (defined within $10^{-7}$ of relative optimality tolerance) or until the time limit of 3600 seconds is reached.

In the following experiments, we compare resolution methods pairwise using a speed-up indicator. For given approaches $m_1$ and $m_2$, the \textit{speed-up} achieved by $m_1$ with respect to $m_2$ on a given instance is the ratio $\frac{CPU(m_2)}{CPU(m_1)}$.
The \textit{average speed-up}, computed on a set $I$ of $p$ instances, is the geometric mean $(\Pi_{i=1}^{p} s_i)^{\frac{1}{p}}$ of the speed-ups $s_1$,..., $s_{p}$.

The following methods are considered:

\begin{tabular}{p{2.5cm}p{10cm}} 
- \textit{Default} Cplex: & Default implementation of Cplex used by its C++ API, \\  
- \textit{Callback} Cplex: & Cplex with empty Branch and LazyConstraint Callbacks,\\
- MOB: & modified orbital branching with no branching rules enforced (Cplex is free to choose the next branching variable),\\
- SOF:  & Static orbitopal fixing for the full-orbitope, \\
- DOF & Dynamic orbitopal fixing for the full orbitope, \\
- SOF-S: & Static orbitopal fixing for the full orbitope and sub-orbitope, \\
- DOF-S: & Dynamic orbitopal fixing for the full orbitope and sub-orbitope. \\
\end{tabular}

 For methods MOB, SOF, DOF, SOF-S and DOF-S, we also use Cplex C++ API. 
 The fixing (or branching) algorithms are included in Cplex using the
so-called Branch Callback, alongside with an empty LazyConstraint Callback to warn Cplex that our methods will remove solutions from the feasible set.
Note that such callbacks deactivate some Cplex features 
designed to improve the efficiency of the overall algorithm. This may
induce a bias when comparing results obtained with and without the use of a callback.
This is why we also compare our methods to Callback Cplex.

%By extension, the speed-up can be defined for the number of nodes $\frac{nodes(m_2)}{nodes(m_1)}$ or for the best feasible solution value $\frac{best(m_2)}{best(m_1)}$.

\subsection{Instances}
\label{sec:inst}

In order to determine which symmetry-breaking technique performs best with respect to the number of rows and columns of matrix $X$, we consider various instance sizes $(n,T)$. Namely, we generate instances with $T=96$ and smaller $n$ :  (30, 96),  (60,96) and instances with $T=48$ and larger $n$: (60, 48), (80,48). 

For each pair $(n,T)$, we generate a set of MUCP instances as follows.

For each instance, we generate a ``2-peak per day" type demand with a large variation between peak and off-peak values: during one day, the typical demand in energy has two peak periods, 
one in the morning and one in the evening.
The amplitudes between peak and off-peak periods have similar characteristics to those in the dataset from \cite{Carrion06}.

We consider the parameters ($P_{min}$, $P_{max}$, $L$, $\ell$, $c_f$, $c_0$, $c_p$) of each unit from the dataset presented in \cite{Carrion06}. We draw a correlation matrix between these characteristics and define a possible range for each characteristic.
In order to introduce symmetries in our instances, some units are randomly generated based on the parameters correlations and ranges. Each unit generated is duplicated $d$ times, where $d$ is randomly selected in $[1, \frac{n}{F}]$ in order to obtain a total of $n$ units.
The parameter $F$ is called symmetry factor, and can vary from 2 to 4 depending on the value of $n$. Note that these instances are generated along the same lines as literature instances considered in \cite{MUCP16}, but with different $F$ factors.

Table \ref{data} provides some statistics on the instances characteristics. For each instance, a group is a set of two or more units with same characteristics. Each unit which has not been duplicated is a singleton. 
The first and second entries column-wise are the number of singletons and groups. The third entry is the mean group size and the fourth entry is the maximum group size.
Each entry row-wise corresponds to the average value obtained over 20 instances with same size $(n,T)$ and same symmetry factor $F$.

\begin{table}[ht]
\centering
\begin{tabular}{|c|c||c|S[table-format=1.1]|S[table-format=2.1]|S[table-format=2.1]|}

\hline
{Size $(n,T)$} & {Sym. factor} & {Nb of singletons} & {Nb of groups} & {Mean size of groups} & {Group max size} \\

\hline
(30,96)  & F = 4& 1.3 & 6.5 & 4.5 & 6.7 \\ 
\cline{2-6}
 & F = 3& 0.4 & 5.3 & 6.0 & 8.7 \\ 
\cline{2-6}
 & F = 2& 0.6 & 4.1 & 7.6 & 11.4 \\ 
\hline
\hline
(60,96)  & $F$ = 4& 0.6 & 7.9 & 7.8 & 13.3 \\ 

\cline{2-6}
 & $F$ = 3& 0.3 & 6.0 & 10.5 & 16.7 \\ 

\cline{2-6}
 & $F$ = 2& 0.2 & 4.4 & 14.8 & 24.9 \\ 

\hline
\hline
(60,48)  & $F$ = 4& 0.8 & 7.7 & 7.9 & 13.1 \\ 
\cline{2-6}
 & $F$ = 3& 0.6 & 5.8 & 10.9 & 17.8 \\ 
\cline{2-6}
 & $F$ = 2& 0.2 & 4.8 & 13.9 & 23.8 \\ 
\hline
\hline
(80,48)  & $F$ = 4& 0.4 & 8.0 & 10.6 & 18.5 \\ 

\cline{2-6}
 & $F$ = 3& 0.5 & 6.7 & 12.5 & 22.2 \\ 

\cline{2-6}
 & $F$ = 2& 0.1 & 4.5 & 18.9 & 31.4 \\ 

\hline

\iffalse
\hline
(20,192)  & F = 4& 2.2 & 5.4 & 3.4 & 4.8 \\ 
\cline{2-6}
 & F = 3& 1.4 & 5.1 & 3.8 & 5.5 \\ 
\cline{2-6}
 & F = 2& 0.7 & 3.7 & 5.4 & 8.1 \\ 
\hline
\fi

\end{tabular}
\caption{Instance characteristics}
\label{data}
\end{table}

Note that the most symmetrical instances are the ones with the highest $\frac{n}{F}$ ratio. Indeed, these instances feature large groups of identical units, and the size of solution orbits grows exponentially with the size of these groups.
It is well-known that symmetries dramatically impair the B\&B solution process. The highly symmetrical instances are thus expected to be the hardest ones. 
We also expect that symmetry-breaking techniques will prove useful specifically on these instances.

\iffalse
\subsection{Improvement score}
 \label{IS}

In the following experiments, we compare various resolution methods.
The experimental results show there is an important variability in the computation time within groups of instances with same size $(n,T)$ and same symmetry factor $F$. 
We thus introduce the CPU time improvement score, which is a performance ratio comparing the CPU times of two methods.
The \textit{improvement score} for two given methods $m_1$ and $m_2$ is defined as:

\[
\begin{array}{lll}
   I_{CPU}(m_1, m_2) = 2 \frac{ CPU(m_2) - CPU(m_1)}{CPU(m_1) + CPU(m_2)} 
\end{array}
\]

We similarly define the node improvement score and the best feasible solution improvement score:

\[
\begin{array}{lll}
   I_{nodes}(m_1, m_2) = 2 \frac{ nodes(m_2) - nodes(m_1)}{nodes(m_1) + nodes(m_2)} 
\end{array}
\]
\[
\begin{array}{lll}
   I_{best}(m_1, m_2) = 2 \frac{best(m_2) - best(m_1)}{best(m_1) + best(m_2)} 
\end{array}
\]

For any indicator $ind$ and any two methods $m_1$ and $m_2$, the considered improvement score $I_{ind}(m_1, m_2)$ provides a symmetric comparison between the two methods $m_1$ and $m_2$.
Indeed, the improvement score is a performance ratio, where the reference used is the average of the indicator values obtained from $m_1$ and $m_2$.
Using this average value as reference yields the following key property: $I_{ind}(m_1, m_2) = - I_{ind}(m_2, m_1)$.
In particular, $I_{ind}(m_1, m_2) \in [-2, 2]$, while the standard relative error calculated as $\frac{ind(m_1) - ind(m_2)}{ind(m_1)} \in [ - \infty, 1] $ would be non-symmetric and unbounded.
\fi

\subsection{Static and dynamic orbitopal fixing}

The average speed-up achieved by DOF over SOF is given in Table \ref{SOF_DOF}. The average is computed for each group of 20 instances with same size and symmetry factor.

 \begin{table}[ht]
\centering
\begin{tabular}{|c|c|c||c|c|c|}
\hline
 \multicolumn{3}{|c||}{(30, 96)} & \multicolumn{3}{c|}{(60, 48)}\\
\hline
  $F=4$ & $F=3$ & $F=2$  & $F=4$ & $F=3$ & $F=2$ \\
\hline
14.5  & 3.6  & 2.6  & 4.6  & 11.7  & 6.7 \\
\hline
\end{tabular}
\caption{Speed-up of DOF with respect to SOF}
\label{SOF_DOF}
\end{table}

It is clear that DOF outperforms SOF on each group, by a factor ranging from 2.6 to 14. Thus, we do not consider SOF nor SOF-S in the following experiments.
This behavior can be explained as DOF allows for more variable fixings earlier in the B\&B tree.
Indeed, the orbitopal fixing algorithm propagates a branching decision occurring at $r^{th}$ row (with respect to the lexicographical order) only if there are enough variables already fixed in $1^{st}$ to $r-1^{th}$ rows.
As DOF defines the lexicographical order with respect to the branching decisions, chances are that many variables are already fixed in each row with rank less than $r$. Thus, DOF often propagates branching decisions in the B\&B tree earlier than SOF does.

%The superiority of DOF over SOF is explained by the fact that it allows more variable fixings earlier in the B\&B tree, compared to SOF. Indeed, in order to be able to propagate a branching decision occurring at row $t$, there must be enough fixed variables in rows 1 to $t-1$. By reordering the rows with respect to the branching decisions, DOF is able to fix variables in row $t$ even if no variable has been fixed in rows 1 to $t-1$.

Note that MOB also follows the branching decisions, as it branches on a whole variable orbit, $i.e.$, a set of symmetrical variables on a given row.
Contrary to DOF, MOB does not account for variables outside the orbit, whereas these variables could be fixed as well.

\iffalse
\begin{table}[ht]
\centering
\begin{tabular}{|c|c||c|c|c||c|c|c||c|c|c|}
\hline
\multicolumn{2}{ |c||}{ (15, 288) } & \multicolumn{3}{c||}{(20, 192)}& \multicolumn{3}{c||}{(30, 96)} & \multicolumn{3}{c|}{(60, 48)}\\
\hline
$F=3$ & $F=2$ & $F=4$ & $F=3$ & $F=2$  & $F=4$ & $F=3$ & $F=2$  & $F=4$ & $F=3$ & $F=2$ \\
\hline
72.7 \%  & 102 \%  & 79.3 \%  & 116 \%  & 65.1 \%  & 52.7 \%  & 84 \%  & 52.5 \%  & 50.3 \%  & 43.3 \%  & 39.5 \%  \\
\hline
\end{tabular}
\caption{$I_{CPU}(SOF, DOF)$}
\label{SOF_DOF}
\end{table}
\fi

\subsection{Modified orbital branching for the MUCP}

The authors in \cite{Ostrowski15} apply MOB alongside with several complementary branching rules to break symmetries of the MUCP with additional technical constraints. Note that sub-symmetries, defined in Section \ref{sec:sub}, appear in the symmetry groups of the subproblems associated to the B\&B nodes. In practice, this is not exploited in \cite{Ostrowski15}, where the symmetries considered at each node are all contained in the symmetry group of the global problem.
Different approaches are compared experimentally: Default Cplex,
 Callback Cplex, OB (orbital branching), MOB with no branching rules enforced (Cplex is free to choose the next branching variable), and MOB with RMRI (the most flexible branching rule ensuring full-symmetry breaking).
 
Because advanced Cplex features are turned off when callbacks are used, there is still a huge performance gap between Callback Cplex and default Cplex.
It is shown in \cite{Ostrowski15} that MOB with RMRI is more efficient than MOB, OB and Callback Cplex in terms of CPU time. The difference between using MOB with RMRI and MOB alone is however not as significant as the difference between MOB and simple orbital branching. 
In particular, referring to the experimental results obtained in \cite{Ostrowski15}, the (geometric) average CPU time speed-up between MOB and MOB+RMRI is $1.098$.
Even though MOB+RMRI is slightly better than MOB with no branching rules, we will choose in Section \ref{sec:comp} to compare our methods to MOB. The rationale behind is that its implementation is straightforward, thus leaving no room to interpretation.

\subsection{Comparison of Cplex, MOB, DOF and DOF-S}
\label{sec:comp}

\begin{table}[!htbp]
\centering
\begin{tabular}{|c|c||c|c|S[table-format=8.2]|S[table-format=7.1, table-align-text-post = false]|S[table-format=5.1]|}
\hline
\multicolumn{2}{|c||}{Instances} & Method & \#opt & {\#nodes} & {\#fixings} & {CPU time} \\
\hline
(30,96)  & $F = $ 4 & DC & 20 & 34742 & {-}  & 34\\
\hhline{~~-----}
 & & CC & 12 & 1669334 & {-}  & 1506\\
\hhline{~~-----}
 & & MOB & 14 & 794522 & 49529 & 1212\\
\hhline{~~-----}
 & & DOF & 19 & 325977 & 135984 & 339\\
\hhline{~~-----}
 & & DOF-S & 20 & 96416 & 74281 & 129\\
\hhline{~======}
 & $F = $ 3 & DC & 16 & 823455 & {-}  & 877\\
\hhline{~~-----}
 & & CC & 8 & 1977613 & {-}  & 2296\\
\hhline{~~-----}
 & & MOB & 12 & 733875 & 197964 & 1578\\
\hhline{~~-----}
 & & DOF & 13 & 831504 & 667733 & 1338\\
\hhline{~~-----}
 & & DOF-S & 16 & 484930 & 564660 & 899\\
\hhline{~======}
 & $F = $ 2 & DC & 17 & 367672 & {-}  & 606\\
\hhline{~~-----}
 & & CC & 11 & 1244729 & {-}  & 1727\\
\hhline{~~-----}
 & & MOB & 12 & 960300 & 660193 & 1525\\
\hhline{~~-----}
 & & DOF & 14 & 575483 & 698740 & 1089\\
\hhline{~~-----}
 & & DOF-S & 17 & 496889 & 736485 & 1026\\
\hline
\hline
(60,96)  & $F = $ 4 & DC & 9 & 1971737 & {-}  & 1994\\
\hhline{~~-----}
 & & CC & 3 & 1899968 & {-}  & 3072\\
\hhline{~~-----}
 & & MOB & 9 & 730306 & 1037813 & 2082\\
\hhline{~~-----}
 & & DOF & 8 & 932314 & 3992329 & 2224\\
\hhline{~~-----}
 & & DOF-S & 10 & 678260 & 3410927 & 1828\\
\hhline{~======}
 & $F = $ 3 & DC & 10 & 1679013 & {-}  & 2134\\
\hhline{~~-----}
 & & CC & 0 & 1890180 & {-}  & 3600\\
\hhline{~~-----}
 & & MOB & 3 & 649769 & 381602 & 3064\\
\hhline{~~-----}
 & & DOF & 5 & 952878 & 1813052 & 2957\\
\hhline{~~-----}
 & & DOF-S & 7 & 633231 & 2193599 & 2465\\
\hhline{~======}
 & $F = $ 2 & DC & 9 & 1669806 & {-}  & 2128\\
\hhline{~~-----}
 & & CC & 0 & 1295402 & {-}  & 3600\\
\hhline{~~-----}
 & & MOB & 7 & 562942 & 281326 & 2490\\
\hhline{~~-----}
 & & DOF & 8 & 496424 & 967275 & 2379\\
\hhline{~~-----}
 & & DOF-S & 8 & 525966 & 1322964 & 2199\\
\hline
\hline
(60,48)  & $F = $ 4 & DC & 17 & 1059290 & {-}  & 830\\
\hhline{~~-----}
 & & CC & 8 & 2664489 & {-}  & 2252\\
\hhline{~~-----}
 & & MOB & 17 & 348881 & 205477 & 639\\
\hhline{~~-----}
 & & DOF & 16 & 665100 & 702066 & 764\\
\hhline{~~-----}
 & & DOF-S & 17 & 431652 & 694538 & 558\\
\hhline{~======}
 & $F = $ 3 & DC & 13 & 1322111 & {-}  & 1283\\
\hhline{~~-----}
 & & CC & 7 & 2224234 & {-}  & 2374\\
\hhline{~~-----}
 & & MOB & 13 & 932987 & 778563 & 1317\\
\hhline{~~-----}
 & & DOF & 15 & 486352 & 972444 & 922\\
\hhline{~~-----}
 & & DOF-S & 15 & 443246 & 1083904 & 935\\
\hhline{~======}
 & $F = $ 2 & DC & 17 & 701617 & {-}  & 645\\
\hhline{~~-----}
 & & CC & 10 & 1448065 & {-}  & 1804\\
\hhline{~~-----}
 & & MOB & 18 & 190009 & 54377 & 417\\
\hhline{~~-----}
 & & DOF & 18 & 150486 & 407031 & 382\\
\hhline{~~-----}
 & & DOF-S & 19 & 135906 & 449141 & 325\\
\hline
\hline
(80,48)  & $F = $ 4 & DC & 8 & 2423226 & {-}  & 2168\\
\hhline{~~-----}
 & & CC & 1 & 2653960 & {-}  & 3420\\
\hhline{~~-----}
 & & MOB & 5 & 1134716 & 1047231 & 2798\\
\hhline{~~-----}
 & & DOF & 6 & 1185164 & 2246156 & 2607\\
\hhline{~~-----}
 & & DOF-S & 9 & 861262 & 2476840 & 2160\\
\hhline{~======}
 & $F = $ 3 & DC & 10 & 1404892 & {-}  & 2015\\
\hhline{~~-----}
 & & CC & 1 & 1553426 & {-}  & 3447\\
\hhline{~~-----}
 & & MOB & 2 & 744775 & 262750 & 3247\\
\hhline{~~-----}
 & & DOF & 2 & 936007 & 1062502 & 3248\\
\hhline{~~-----}
 & & DOF-S & 3 & 865991 & 1285128 & 3169\\
\hhline{~======}
 & $F = $ 2 & DC & 8 & 2715484 & {-}  & 2217\\
\hhline{~~-----}
 & & CC & 0 & 3628624 & {-}  & 3600\\
\hhline{~~-----}
 & & MOB & 6 & 1145092 & 1150613 & 2552\\
\hhline{~~-----}
 & & DOF & 6 & 1594025 & 3597266 & 2591\\
\hhline{~~-----}
 & & DOF-S & 8 & 1328985 & 2662087 & 2269\\
\hline
\end{tabular}
\captionsetup{justification=centering}
\caption{Performance indicators relative to the comparison of five methods \\ to solve MUCP instances wih symmetries }
\label{average}
\end{table}

We compare five different resolution methods for the MUCP: Default Cplex, Callback Cplex, MOB, DOF and DOF-S. 
As shown in Table \ref{SOF_DOF}, dynamic orbitopal fixing outperforms the static variant, thus SOF and SOF-S are not considered.

Table \ref{average} provides, for each method and each group of 20 instances:

\begin{tabular}{p{1.5cm}p{12cm}} 
\#opt: & Number of instances solved to optimality,   \\  
\#nodes: & Average number of nodes, \\
\#fixings: & Average number of fixings  (for MOB, it is the total number of variables fixed during the branching process) \\
CPU time: & Average CPU time in seconds. \\
\end{tabular}

Note that the best feasible solution value is not reported, as all methods are able to find the same best feasible solution value within the time limit.

First note that instances of size (80,48) and, to a lesser extent, of size (60, 96), are the hardest ones: Default Cplex only solves to optimality half of them, and Callback Cplex solves nearly none of them.  Further increases in the number $n$ of units or in the number $T$ of time steps would then not be of particular interest, if the corresponding instances are intractable.
%then probably yield instances that Default Cplex cannot manage at all.

Interestingly, increasing the number $n$ of units seems to have more impact on the CPU time than increasing the number $T$ of time steps.
Indeed, from instances of size (60,48) to instances of size (80,48), $n$ is only multiplied by a factor 1.3, but the computation time increases by a factor 2. A similar increase in computation time is obtained from instances of size (60,48) to instances of size (60,96), but in this case the number $T$ of time periods has increased by a factor 2.
Similarly, from instances of size (30,96) to instances of size (60,48), $n$ increases but $T$ decreases, and both the CPU time and the number of nodes increase.
This strong computational impact of parameter $n$ illustrates the polynomiality of the MUCP when $n$ is fixed and $T$ is arbitrary \cite{ComplexityMUCP}.

Note that in average, MOB explores more nodes in comparison with DOF and DOF-S. Even though MOB has more opportunities to fix variables due to the large number of nodes visited, the number of fixings performed by DOF or DOF-S is always much larger (often by at least one order of magnitude). 
Thus, DOF and DOF-S solve MUCP instances faster, since they branch less thanks to the fixing procedure.

\paragraph{}
Table \ref{global_table}  compares each method $m_1$, among MOB, DOF and DOF-S,  with respect to method $m_2$, among Default Cplex and Callback Cplex,  in terms of average speed-up. The average speed-up is computed on groups of 20 instances of same size $(n,T)$ and same symmetry factor $F$, as described in Section \ref{sec:inst}.

Table  \ref{global_table} shows:

\begin{tabular}{p{1.5cm}p{12cm}} 
$(n,T)$: & Instance size, \\
$F$: & Symmetry factor, \\
$m_1$: & Method $m_1$, namely MOB, DOF or DOF-S, \\
$m_2$: & Method $m_2$, namely Default Cplex or Callback Cplex, \\
\#opt: & Number of instances solved to optimality by $m_1$, \\
opt$_\Delta$: & Difference in terms of the number of instances solved to optimality by $m_1$ and by $m_2$, \\
$S_{CPU}$: & Average speed-up by method $m_1$ with respect to $m_2$, computed on a group of 20 instances. \\
\end{tabular}

\begin{table}[ht]
\centering
\begin{tabular}{|c|c||c|c||cH|S[table-format=3.3] ||cH|S[table-format=3.2]|}
\hline
\multicolumn{2}{|c||}{Instance} & \multicolumn{2}{c||}{} & \multicolumn{3}{c||}{$m2 = $ Default Cplex} & \multicolumn{3}{c|}{$m_2 = $ Callback Cplex} \\
\hline
$(n,T)$ & Sym & $m_1$ & \#opt & opt$_{\Delta}$ &  $S_{nodes}$ & $S_{CPU}$ & opt$_{\Delta}$  & $S_{n}$ & $S_{cpu}$ \\
\hline 
(30,96)  & $F = $ 4 & MOB & 14 & -6 & 0.192  & 0.0902  & 2 & 3.22  & 1.57  \\
 & & DOF & 19 & -1 & 0.646 & 0.659  & 7 & 11.8 & 11.4  \\
 & & DOF-S & 20 & 0 & 0.55& 0.725  & 8 & 14.2& 12.6  \\
\hhline{~---------}
 & $F = $ 3 & MOB & 12 & -4 & 0.504  & 0.371  & 4 & 22.2  & 3.78  \\
 & & DOF & 13 & -3 & 0.359 & 0.507  & 5 & 8.5 & 5.17  \\
 & & DOF-S & 16 & 0 & 1.03& 1.05  & 8 & 14.5& 10.7  \\
\hhline{~---------}
 & $F = $ 2 & MOB & 12 & -5 & 0.926  & 0.197  & 1 & 12.2  & 2.1  \\
 & & DOF & 14 & -3 & 2.77 & 0.564  & 3 & 38.7 & 6  \\
 & & DOF-S & 17 & 0 & 1.57& 0.716  & 6 & 50.5& 7.62  \\
\hline
 \hline
(60,96)  & $F = $ 4 & MOB & 2 & -8 & 0.121  & 0.218  & 1 & 30.3  & 1.36  \\
 & & DOF & 2 & -8 & 0.0459 & 0.214  & 1 & 15.9 & 1.33  \\
 & & DOF-S & 3 & -7 & 0.0434& 0.218  & 2 & 51.1& 1.36  \\
\hhline{~---------}
 & $F = $ 3 & MOB & 3 & -7 & 0.504  & 0.314  & 3 & 1  & 2.33  \\
 & & DOF & 5 & -5 & 0.0753 & 0.244  & 5 & 1 & 1.81  \\
 & & DOF-S & 7 & -3 & 1.59& 0.555  & 7 & 1& 4.11  \\
\hhline{~---------}
 & $F = $ 2 & MOB & 7 & -2 & 1.82  & 0.358  & 7 & 1  & 3.92  \\
 & & DOF & 8 & -1 & 1.28 & 0.493  & 8 & 1 & 5.39  \\
 & & DOF-S & 8 & -1 & 2.41& 0.669  & 8 & 1& 7.32  \\
\hline
\hline
(60,48)  & $F = $ 4 & MOB & 17 & 0 & 2.99  & 0.978  & 9 & 13.5  & 13.5  \\
 & & DOF & 16 & -1 & 2.9 & 1.45  & 8 & 15.8 & 20.1  \\
 & & DOF-S & 17 & 0 & 3.53& 1.92  & 9 & 18.2& 26.5  \\
\hhline{~---------}
 & $F = $ 3 & MOB & 13 & 0 & 3.86  & 0.94  & 6 & 16.4  & 8.6  \\
 & & DOF & 15 & 2 & 4.72 & 2.07  & 8 & 21.4 & 18.9  \\
 & & DOF-S & 15 & 2 & 4.82& 2.25  & 8 & 21.9& 20.6  \\
\hhline{~---------}
 & $F = $ 2 & MOB & 18 & 1 & 3  & 1.84  & 8 & 6.73  & 11.7  \\
 & & DOF & 18 & 1 & 1.88 & 2.58  & 8 & 4.1 & 16.4  \\
 & & DOF-S & 19 & 2 & 1.47& 2.6  & 9 & 4.1& 16.5  \\
\hline 
\hline 
(80,48) & $F = $ 4 & MOB & 5 & -3 & 1.09  & 0.316  & 4 & 257  & 2.88  \\
 & & DOF & 6 & -2 & 0.311 & 0.462  & 5 & 392 & 4.21  \\
 & & DOF-S & 9 & 1 & 0.667& 0.75  & 8 & 392& 6.83  \\
 \hhline{~---------}
 & $F = $ 3 & MOB & 6 & -2 & 25.2  & 0.637  & 6 & 1  & 4.97  \\
 & & DOF & 6 & -2 & 0.919 & 0.422  & 6 & 1 & 3.29  \\
 & & DOF-S & 8 & 0 & 4.61& 0.792  & 8 & 1& 6.18  \\
\hhline{~---------}
 & $F = $ 2 & MOB & 9 & 0 & 1.19  & 0.701  & 6 & 6.19  & 5.22  \\
 & & DOF & 8 & -1 & 1.03 & 0.632  & 5 & 6.24 & 4.7  \\
 & & DOF-S & 10 & 1 & 2.49& 1.1  & 7 & 7.73& 8.14  \\
 \hline 

\iffalse
\hline 
(20,192)  & $F = $ 4 & MOB & 17 & -2 & 0.118  & 0.129  & 1 & 2.51  & 1.12  \\
 & & DOF & 19 & 0 & 0.319 & 0.514  & 3 & 8.5 & 4.46  \\
 & & DOF-S & 19 & 0 & 0.666& 0.865  & 3 & 13& 7.51  \\
\hhline{~---------}
 & $F = $ 3 & MOB & 13 & -5 & 0.089  & 0.118  & 1 & 1.36  & 1.31  \\
 & & DOF & 14 & -4 & 0.232 & 0.21  & 2 & 5.31 & 2.33  \\
 & & DOF-S & 16 & -2 & 0.382& 0.391  & 4 & 10.9& 4.35  \\
\hhline{~---------}
 & $F = $ 2 & MOB & 12 & -6 & 0.578  & 0.183  & 0 & 3.15  & 1.69  \\
 & & DOF & 13 & -5 & 1.08 & 0.236  & 1 & 7.89 & 2.18  \\
 & & DOF-S & 13 & -5 & 1.65& 0.286  & 1 & 9.49& 2.64  \\
\hline
 \fi
\end{tabular}
\captionsetup{justification=centering}
\caption{MOB and dynamic orbitopal fixing (DOF and DOF-S) - average speed-up for various instances compared to Default Cplex and Callback Cplex}
\label{global_table}
\end{table}

In terms of CPU time, MOB, DOF and DOF-S greatly outperform Callback Cplex, but the improvement is larger with DOF and even more significant with DOF-S. 
Indeed, even on the less symmetrical instances ($(n,T)=(30,96)$ and $F=4$), MOB outruns Callback Cplex by a factor 1.57 and DOF increases this factor to 11.4.
Similarly, on more symmetrical instances $(n,T)=(60,48)$, $F=4$ (resp. $F=3$, $F=2$), MOB outperforms Callback Cplex by a factor 13.5 (resp. 8.6, 11.7) while DOF increases this factor to 20.1 (resp. 18.9, 16.4).

When both symmetries and sub-symmetries are accounted for, the performance is significantly improved.
For example, on some of the less symmetrical instances ($(n,T)=(30,96)$ and $F=3$), DOF outruns Callback Cplex by a factor 5.17 and this factor increases to 10.7 with DOF-S.
Similarly, on more symmetrical instances $(n,T)=(60,96)$, $F=3$ (resp. $F=2$), DOF outperforms Callback Cplex by a factor 1.81 (resp. 5.39) while DOF-S increases this factor to 4.11 (resp. 7.32). On instances $(n,T)=(60,48)$, $F=4$, DOF-S is even faster than Callback Cplex by a factor 26.5.

\paragraph{}

As observed in \cite{Ostrowski15}, there is a huge performance gap between Callback Cplex and Default Cplex.
Thus, even if MOB, DOF and DOF-S substancially outperforms Callback Cplex in each instance group, it is sometimes not enough to close the performance gap between Default and Callback Cplex, especially for instances with small $n$. On the opposite, for large $n$ instances where symmetries are a major source of difficulty, DOF and DOF-S clearly outperforms Default Cplex.

Typically, when $T$ is large compared to $n$ ($i.e.$, on instances of size (60,96) and (30,96)) it seems that non symmetry-related difficulties arise, and none of the compared methods catch up with Default Cplex. In this context, the cost of applying symmetry-breaking techniques (including the performance loss induced by the use of a Callback) seems too important compared to the impact of symmetries.
The performance loss is less important with DOF and DOF-S than it is with MOB.
DOF-S is the method that is the closest to catch up with Default Cplex. Indeed, for $(n,T)=(30,96)$ instances, it solves to optimality as many instances as Default Cplex,
and on $F=3$ instances of size (30,96) DOF-S even slightly improves Default Cplex, while MOB is slower than Default Cplex by a factor 3.

On the opposite, when $n$ is large compared to $T$ ($i.e.$, on instances of size (80,48) and (60,48)), symmetry seems to be a major factor of computational difficulty. Indeed, DOF-S performs quite well in this context and solves to optimality some instances Default Cplex cannot. For example, on instances $(n,T)=(60,48)$, $F=2$ (resp. $F=3$), DOF-S solves two more instances to optimality than Default Cplex.
DOF and MOB do not perform as well as DOF-S in this respect.
On instances of size (60,48), DOF and DOF-S outrun Default Cplex by a factor 2, while MOB is closer to a factor 1. When $n$ increases to 80, DOF-S achieves a speed-up of 1.1 compared to Default Cplex on the most symmetrical instances ($F=2$), while MOB and DOF stay behind with a speed-up around 0.7 relatively to Default Cplex. Moreover, DOF-S solves more instances to optimality than Default Cplex. For less symmetrical instances with $n=80$, $i.e.$  $F=3$ and $F=4$ groups, none of the compared methods are able to outrun Default Cplex in terms of CPU time. It seems that non-symmetry related difficulties inherent to the MUCP arise in these instances featuring a large number of distinct units.
In this context, DOF-S is the method closest to catch up with Default Cplex. Indeeed, on both groups of instances, the speed-up provided by DOF is around 0.8, whereas this factor ranges from 0.3 to 0.6 for MOB and DOF.
While Callback Cplex solves to optimality only one instance out of forty, DOF-S proves its efficiency by solving even more instances to optimality than Default Cplex.

\section*{Conclusion}
 
In this paper, we define a linear time orbitopal fixing algorithm for 
the full orbitope. 
We propose to also account for symmetries arising in solutions subsets, which we refer to as sub-symmetries. Sub-symmetries related to sub-symmetric groups are considered, 
leading us to define the full sub-orbitope. 
We extend our orbitopal fixing algorithm in order to apply it to both orbitope and sub-orbitope structures.
This algorithm is proven to be optimal, in the sense that at any node $a$ in the search tree, any variable that 
can be fixed, with respect to the lexicographical order, is fixed by the algorithm.
 %From the simple observation that when more variables on small-index rows  are already fixed at node $a$, more fixings can be performed.
 We propose a dynamic version of the 
orbitopal fixing algorithm,
where the lexicographical order at node $a$ is defined with respect to the 
branching decisions leading
to $a$.

For MUCP instances, experimental results show that the dynamic variant of our algorithm performs much better than the static variant. 
  Moreover, it is clear that sub-symmetries greatly impair the solution process for MUCP instances, since dynamic orbitopal fixing for both full orbitope and full sub-orbitope (DOF-S) performs even better than dynamic orbitopal fixing for the full orbitope (DOF).
  Finally, our experiments show that our approach is 
 competitive with commercial solvers like Cplex and state-of-the-art
 techniques like modified orbital branching (MOB).
 Even if MOB already improves Callback Cplex, the improvement is even more significant with our methods DOF and DOF-S. Furthermore, even if there is a huge performance gap between Callback Cplex and Default Cplex,  DOF-S is able to outrun Default Cplex by a factor 2 on some of the most symmetrical instances.

In the past, the complete linear description of partitioning and 
packing orbitopes helped to
design an orbitopal fixing algorithm for these orbitopes.
Likewise in the future, the orbitopal fixing
algorithm for the full orbitope, by improving our understanding of this 
polyhedron, might help to
find a complete linear description of the full orbitope.
Moreover, it 
would be interesting to extend
orbitopal fixing to full orbitopes under other group actions, for 
example the cyclic group. 
Another approach to handle symmetries related to the symmetric or the cyclic group would be to find a new set of representatives whose convex hull would be easier to describe than the full orbitope. Another promising perspective would be to adapt existing symmetry-breaking techniques to break sub-symmetries as well, in the case of arbitrary sub-symmetry groups.

Finally, there is a wide range of problems featuring all column permutation symmetries and sub-symmetries, in particular many variants of the UCP, on which it would be desirable to analyze the effectiveness of our approach. Other examples of such problems can be found among covering problems, whose solution matrix has at least one 1-entry per row, like bin-packing variants. Even though computing the exact fixing has been shown NP-hard in this case, our orbitopal fixing algorithm, designed for full orbitopes, can be used to compute valid variable fixings in a covering orbitope as well. 
%In this case, there is no guarantee that fixings are done as early as possible in the tree, because the special structure of covering orbitopes may induce possible fixings that would not be correct in a full orbitope. Nevertheless, this fixing algorithm breaks all column-permutation related symmetries at some point in the B\&B tree, which may be sufficient to overcome the computational difficulties arising from the highly symmetric nature of these problems.

%\begin{acknowledgements}
%\end{acknowledgements}

\bibliographystyle{plain}      % mathematics and physical sciences

\bibliography{paper}

\end{document}